\documentclass[a4paper,reqno,11pt]{amsart}
\usepackage{geometry}
\usepackage[T1]{fontenc}
\usepackage[utf8]{inputenc}
\usepackage{xspace,
	amsfonts}
\usepackage{amsmath}
\usepackage{amssymb}
\usepackage{graphicx,
	bbm,
	tikz,
	subfig}
\usepackage{xfrac}
\usepackage{dsfont}
\usepackage{amsthm}
\usepackage{mathtools,mathrsfs}
\usepackage{enumitem}
\usepackage{verbatim}
\usepackage[autostyle]{csquotes}

\usepackage{cite}

\mathtoolsset{showonlyrefs}

\usepackage{xcolor}

\theoremstyle{plain}
\newtheorem{thm}{Theorem}[section]

\newtheorem{lem}[thm]{Lemma}
\theoremstyle{definition}
\newtheorem{defn}[thm]{Definition}
\newtheorem{rem} [thm] {Remark}

\newcommand{\R}{\mathbb{R}}
\newcommand{\Z}{\mathbb{Z}}
\newcommand{\N}{\mathbb{N}}
\newcommand{\V}{\mathbb{V}}
\newcommand{\E}{\mathcal{E}}

\newcommand{\G}{\mathcal G}
\newcommand{\vv}{\textsc{v}}

\usepackage{tikz}
\usetikzlibrary{shapes.geometric,calc,decorations.markings,math}

\tikzstyle{nodo}=[circle,draw,fill,inner sep=0pt,minimum size=%
1.5mm]
\tikzstyle{infinito}=[circle,inner sep=0pt,minimum size=0mm]
\tikzset{every loop/.style={min distance=10mm,in=300,out=240,looseness=10}}

\tikzset{place/.style={circle,thick,draw=blue!75,fill=blue!20,minimum
		size=6mm}}
\tikzset{place2/.style={circle,thick,draw=red!75,fill=red!20,minimum
		size=6mm}}

\title[Doubly nonlinear Schr\"odinger normalized ground states on 2D grids]{Doubly nonlinear Schr\"odinger normalized ground states on 2D grids: existence results and singular limits}

\author[ ]{Daniele Barbera}
\address[D. Barbera]{Politecnico di Torino, Dipartimento di Scienze Matematiche ``G.L. Lagrange'' Corso Duca degli Abruzzi 24, 10129, Torino, Italy.}
\email{daniele.barbera@polito.it}

\author[ ]{Filippo Boni}
\address[F. Boni]{Scuola Superiore Meridionale, Largo S. Marcellino, 10, 80138, Napoli, Italy.}
\email{f.boni@ssmeridionale.it}

\author[ ]{Simone Dovetta}
\address[S. Dovetta]{Politecnico di Torino, Dipartimento di Scienze Matematiche ``G.L. Lagrange'', Corso Duca degli Abruzzi 24, 10129, Torino, Italy.}
\email{simone.dovetta@polito.it}

\author[ ]{Lorenzo Tentarelli}
\address[L. Tentarelli]{Politecnico di Torino, Dipartimento di Scienze Matematiche ``G.L. Lagrange'', Corso Duca degli Abruzzi 24, 10129, Torino, Italy.}
\email{lorenzo.tentarelli@polito.it}

\begin{document}

\begin{abstract}
	We investigate the existence and the singular limit of normalized ground states for focusing doubly nonlinear Schr\"odinger equations with both standard and concentrated nonlinearities on two-dimensional square grids. First, we provide existence and non-existence results for such ground states depending on the values of the nonlinearity powers and on the structure of the set of vertices where the concentrated nonlinearities are located. Second, we prove that suitable piecewise-affine extensions of such states converge strongly in $H^1(\R^2)$ to ground states of corresponding doubly nonlinear models defined on the whole plane as the length of the edges in the grid tends to zero. This convergence is proved both for limit models with standard nonlinearities only and for models combining standard and singular nonlinearities concentrated on a line or on a strip.
\end{abstract}

\maketitle

\vspace{-.5cm}


\section{Introduction}
In the present paper we study the existence and the singular limit of ground states for the doubly nonlinear Schr\"odinger energy functional
\begin{equation}
\label{eq:EG}
	E(u,\G_\varepsilon):=\frac12\|u'\|_{L^2(\G_\varepsilon)}^2-\frac{\alpha}p\|u\|_{L^p(\G_\varepsilon)}^p-\frac\beta q\sum_{\vv\in V}|u(\vv)|^q,
\end{equation}
where, for every $\varepsilon>0$, $\G_\varepsilon=(\V_{\G_\varepsilon},\mathbb E_{\G_\varepsilon})$ is the two-dimensional metric grid with edgelength $\varepsilon$ given by the subset of $\R^2$ with vertices on $\varepsilon\Z^2$ and edges between every couple of vertices at distance $\varepsilon$ (see Figure \ref{fig:grid}),  and $V\subset \V_{\G_\varepsilon}$ is a fixed subset of its vertices.  

A ground state of $E(\cdot,\G_\varepsilon)$ with mass $\mu>0$ is a function $u\in H_\mu^1(\G_\varepsilon)$ such that
\begin{equation}
	\label{eq:levEG}
	E(u,\G_\varepsilon)=\inf_{u\in H_\mu^1(\G_\varepsilon)}E(u,\G_\varepsilon)=:\E_{\G_\varepsilon}(\mu)\,,
\end{equation}
where
\[
H^1_\mu(\G_\varepsilon):=\left\{u\in H^1(\G_\varepsilon)\,:\, \|u\|_{L^2(\G_\varepsilon)}^2=\mu\right\}.
\]
Since $E(\cdot,\G_\varepsilon)$ combines a standard nonlinearity with concentrated nonlinearities located at the vertices in $V$, (now) classical arguments show that any ground state $u\in H_\mu^1(\G_\varepsilon)$ is a positive (up to a change of sign) solution of the following stationary nonlinear Schr\"odinger equation on $\G_\varepsilon$
\[
\begin{cases}
	-u''+\lambda u=\alpha|u|^{p-2}u & \forall\, e\in\mathbb{E}_{\G_\varepsilon}\\[.2cm]
	\displaystyle\sum_{e\succ\vv}u_e'(\vv)=0 & \forall\,\vv\in\V_{\G_\varepsilon}\setminus V\\[.4cm]
	\displaystyle\sum_{e\succ\vv}u_e'(\vv)=-\beta|u(\vv)|^{q-2}u(\vv) & \forall\,\vv\in V
\end{cases}
\]
with nonlinear $\delta$-type conditions at the vertices of $V$ and homogeneous Kirchhoff conditions at all other vertices, for a suitable Lagrange multiplier  $\lambda\in\R$ associated to the so-called mass constraint (i.e., the constraint on the $L^2$-norm).

In what follows, we consider positive parameters $\alpha,\beta>0$  and the purely $L^2$-subcritical regime of powers, namely
\[
2<p<6\,,\quad\,2<q<4\,.
\]
In this setting, our aim is twofold. First, we discuss the dependence of the existence of ground states on $p,q,\mu$ and on the structure of the set of vertices $V$ affected by the concentrated nonlinearity. Second, we investigate the asymptotic behaviour of such ground states on $\G_\varepsilon$ in the singular limit $\varepsilon\to0$.

\begin{figure}[t]
	\begin{center}
		\begin{tikzpicture}[xscale= 0.5,yscale=0.5]
			\draw[step=2,thin] (0,0) grid (8,8);
			\foreach \x in {0,2,...,8} \foreach \y in {0,2,...,8} \node at (\x,\y) [nodo] {};
			\foreach \x in {0,2,...,8}
			{\draw[dashed,thin] (\x,8.2)--(\x,9.2) (\x,-0.2)--(\x,-1.2) (-1.2,\x)--(-0.2,\x)  (8.2,\x)--(9.2,\x); }
			\draw[->,thin] (-.5,1.3)--(-.5,1.9);
			\draw[->,thin] (-.5,.7)--(-.5,.1);
			\node at (-.5,1) [infinito] {{\footnotesize $\varepsilon$}};
			\draw[->,thin] (1.3,-.5)--(1.9,-.5);
			\draw[->,thin] (.7,-.5)--(.1,-.5);
			\node at (1,-.5) [infinito] {{\footnotesize $\varepsilon$}};
		\end{tikzpicture}
	\end{center}
	\caption{The grid  $\G_\varepsilon$.}
	 \label{fig:grid}
\end{figure}
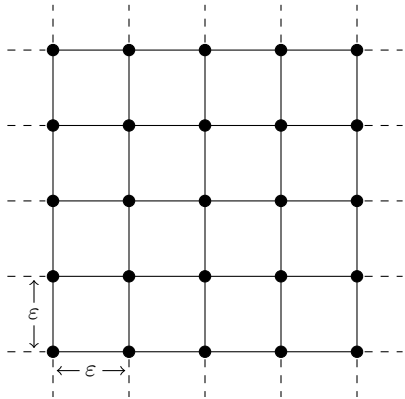

\smallskip
Two-dimensional grids are specific examples of metric graphs, i.e. locally one-dimensional structures obtained by gluing together several (possibly, infinitely many) intervals through the identification of some of their endpoints. The study of nonlinear Schr\"odinger models on metric graphs has been gathering a significant attention in the last years and it is by now a rather active research field. Even though a rich body of literature is nowadays available for models with standard nonlinearities only, that is $\beta=0$ in \eqref{eq:EG} (see e.g. \cite{AST1, AST2, AST3, ACT, ACT2, BMP, BDL1, BDL2, BCJS, CJS, DDGS, DDL, KMPX, KNP, NP, PS, PSV} and references therein), the analysis of doubly nonlinear models involving also $\delta$-type nonlinearities has been started only recently on graphs with finitely many edges in \cite{ABD22, BD22, BD21, LX} (see also \cite{ACFN1, ACFN2, BC23, DS25} for models combining standard nonlinearities with linear concentrated terms). 

Among metric graphs, infinite periodic ones are somehow peculiar, as they combine the typical one-dimensional microscale of metric graphs (the scale of single edges) with a high-dimensional macroscale determined by the degree of periodicity of the structure. This is clearly seen e.g. when thinking of the grid $\G_\varepsilon$ with $\varepsilon\sim0$, that gives a fine approximation of the whole plane $\R^2$ made of one-dimensional intervals.
Such specific co-existence of scales with different dimensions suggests the potential of periodic graphs to serve as a general tool to approximate high-dimensional models posed in full Euclidean spaces with suitable one-dimensional counterparts. Concretely, if it were possible to show that the solutions of a certain problem on the grid $\G_\varepsilon$ are close (in some sense) to those of a limit problem in $\R^2$ when the length of the edges $\varepsilon$ is sufficiently small, then one would obtain a theoretical bridge between the two models that would allow to conveniently switch from one to the other.

In the context of nonlinear Schr\"odinger equations, the validity of this approximation scheme has been confirmed recently for the ground states of the functional with the sole standard nonlinearity
\begin{equation}
\label{eq:b=0}
\frac12\|u'\|_{L^2(\G_\varepsilon)}^2-\frac1p\|u\|_{L^p(\G_\varepsilon)}^p\,.
\end{equation}
The existence of ground states at fixed mass for this problem has been settled in \cite{ADST19}, where the model was already shown to exhibit a mixture of purely one-dimensional and two-dimensional features (see also \cite{AD, ADR, DT} for analogous results in similar settings), and then in \cite{D24} it has been proved that, for every $p\in(2,4)$ and $\mu>0$, suitable extensions to $\R^2$ of properly scaled sequences of ground states $u_\varepsilon$ of \eqref{eq:b=0} converge to the ground states in $H_\mu^1(\R^2)$ of the limit functional
\[
\frac12\|\nabla u\|_{L^2(\R^2)}^2-\frac1p\|u\|_{L^p(\R^2)}^p\,.
\]
Our main goal in this work is to push forward this kind of study on grids, extending the scope of the analysis to \eqref{eq:EG} involving also concentrated nonlinearities.

\smallskip
The first step of this program requires to develop a general existence theory for ground states of \eqref{eq:EG} on two-dimensional grids. In particular, existence of ground states is far from obvious since it is highly sensitive to the interplay between the two nonlinearities and the specific structure of the set $V$ of the vertices carrying the concentrated nonlinearity. In fact, both the actual values of the nonlinearity powers and the set $V$ have already been shown to play a crucial role in determining existence of ground states for the models with a single nonlinearity, in the already mentioned paper \cite{ADST19} for the standard nonlinearity only (i.e. $\beta=0$ in \eqref{eq:EG}), and in the recent work \cite{BDS23} for the concentrated nonlinearity only (i.e. $\alpha=0$ in \eqref{eq:EG}). The existence results we report here provide an extension of these former analyses to the doubly nonlinear setting.

As for the set $V$, we consider both the case of finitely many nonlinear vertices (i.e. $\#V<+\infty$) and that of infinitely many ones (i.e. $\#V=+\infty$). In the latter, since $V$ will clearly be noncompact, it is evident that there is no chance to restore compactness without further assumptions. To this extent, perhaps the most natural thing to do in this context is to explore the possible periodicity of the set $V$. In particular, we will consider $\Z$-periodic and $\Z^2$-periodic subsets of vertices, defined as follows.
\begin{defn}[$\Z$-periodic set $V$]
	\label{def:Z-per}
	A subset $V\subset\V_{\G_1}$ is called $\Z$-periodic (Figure \ref{fig:per}(\small\textsc{A})) if there exists a vector $\vec{v}\in \Z^{2}\setminus\{(0,0)\}$ such that
	\begin{itemize}
		\item[$(i)$] $V=V+ k\vec{v}$, for every $k\in \Z$, and
		\item[$(ii)$] there exist $P_0\in \R^2$ and $r>0$ such that $|(\vv-P_0)\cdot \vec{v}^\perp|\leq r$ for every $\vv\in V$.
	\end{itemize}
	For $\varepsilon\neq1$, a subset $V\subset\V_{\G_\varepsilon}$ is called $\Z$-periodic if $V=\varepsilon V'$, for some $\Z$-periodic set $V'\subset\V_{\G_1}$.
\end{defn}

\begin{defn}[$\Z^2$-periodic set $V$]
	\label{def:Z2-per}
	A subset $V\subset\V_{\G_1}$ is called $\Z^{2}$-periodic (Figure \ref{fig:per}(\small\textsc{B})) if there exist two linearly independent vectors $\vec{v}_1, \vec{v}_2\in \Z^{2}\setminus\{(0,0)\}$ such that
	\begin{equation*}
		V=V+ k_{1}\vec{v}_1+ k_{2}\vec{v}_2\qquad\forall\,k_{1},k_{2}\in \Z\,.
	\end{equation*}
	For $\varepsilon\neq1$, a subset $V\subset\V_{\G_\varepsilon}$ is called $\Z^2$-periodic if $V=\varepsilon V'$, for some $\Z^2$-periodic set $V'\subset\V_{\G_1}$.
\end{defn}

\begin{figure}[t]
	\centering
	\subfloat[ ]{\includegraphics[width=0.4\columnwidth]{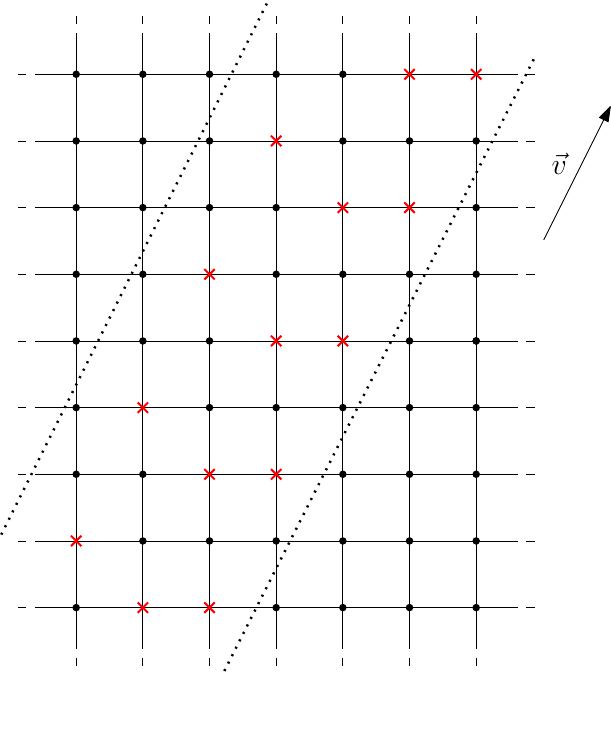}}
	\quad\quad
	\subfloat[ ]{\includegraphics[width=0.4\columnwidth]{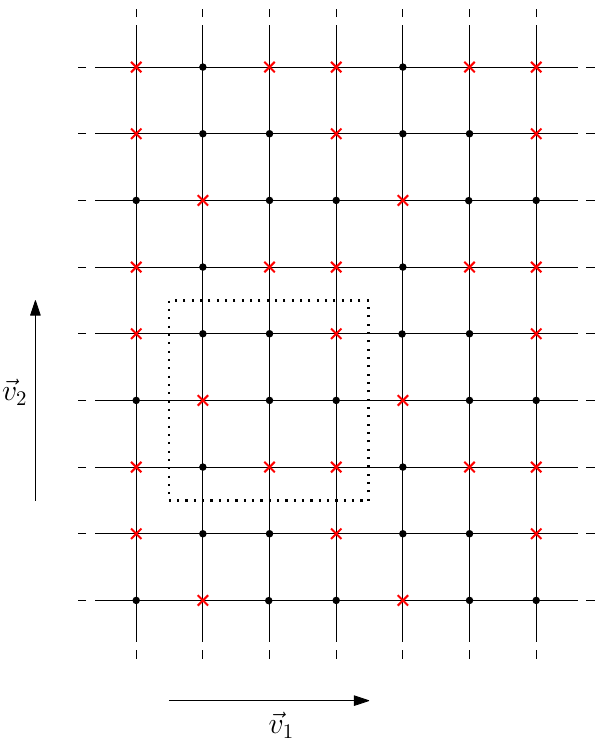}}
	\caption{Examples of a $\Z$-periodic ({\small\textsc{A}}) and a $\Z^2$-periodic ({\small\textsc{B}}) subset $V$ of vertices in a two-dimensional grid (the vertices in $V$ are denoted by red crosses).}
	\label{fig:per}
\end{figure}
We can then collect our main existence results for ground states of \eqref{eq:EG} in the next three theorems, considering finite, $\Z$-periodic and $\Z^2$-periodic sets $V$, respectively.
\begin{thm}
	\label{thm:GScomp}
	Let $p\in(2,6)$, $q\in(2,4)$, $\alpha,\beta>0$ and $V\subset\V_{\G_\varepsilon}$ be such that $\# V<+\infty$. There results that:
	\begin{itemize}
		\item[(i)] if $p\in(2,4)$, then $\E_{\G_\varepsilon}(\mu)<0$ and ground states exist for every $\mu>0$;
		\item[(ii)] if $p\in[4,6)$, then there exists $\overline\mu:=\overline\mu(p,q,\alpha,\beta,V, \varepsilon)>0$ such that
		\[
		\E_{\G_\varepsilon}(\mu)\begin{cases}
		=0 & \text{if }\mu\in(0,\overline\mu]\\[.2cm]
		<0 & \text{if }\mu>\overline\mu\,,
		\end{cases}
		\]
		and ground states exist if $\mu>\overline\mu$ and do not exist if $\mu\in(0,\overline\mu)$.
	\end{itemize}
\end{thm}
\begin{thm}
	\label{thm:GSZ}
	Let $p\in(2,6)$, $q\in(2,4)$, $\alpha,\beta>0$ and $V\subset\V_{\G_\varepsilon}$ be $\Z$-periodic. There results that:
	\begin{itemize}
		\item[(i)] if $p\in(2,4)$ or $q\in(2,3)$, then $\E_{\G_\varepsilon}(\mu)<0$ and ground states exist for every $\mu>0$;
		\item[(ii)] if $p\in[4,6)$ and $q\in[3,4)$, then there exists $\overline\mu:=\overline\mu(p,q,\alpha,\beta, V, \varepsilon)>0$ such that
		\[
		\E_{\G_\varepsilon}(\mu)\begin{cases}
		=0 & \text{if }\mu\in(0,\overline\mu]\\[.2cm]
		<0 & \text{if }\mu>\overline\mu\,,
		\end{cases}
		\]
		and ground states exist if $\mu>\overline\mu$ and do not exist if $\mu\in(0,\overline\mu)$.
	\end{itemize}
\end{thm}
\begin{thm}
	\label{thm:GSZ2}
	Let $p\in(2,6)$, $q\in(2,4)$, $\alpha,\beta>0$ and $V\subset\V_{\G_\varepsilon}$ be $\Z^2$-periodic. Then $\E_{\G_\varepsilon}(\mu)<0$ and ground states exist for every $\mu>0$.
\end{thm}

Theorems \ref{thm:GScomp}-\ref{thm:GSZ}-\ref{thm:GSZ2} highlight a general feature of the doubly nonlinear ground state problem on grids: either ground states exist for every mass, or a threshold phenomenon occurs and ground states with small masses do not exist. Moreover, except possibly at the thresholds, ground states exist if and only if the ground state level is strictly negative. Observe also that existence of ground states is more likely when the set of nonlinear vertices is somehow ``more periodic'', as the region in the $pq$-plane where ground states exist becomes larger and larger as passing from finite to $\Z$-periodic to $\Z^2$-periodic sets $V$. 

These are in fact the typical traits of NLS ground state problems on periodic graphs. Indeed, comparing the results above with those  in \cite{ADST19, BDS23} shows that the qualitative picture is the same as that for models with a single nonlinearity. However, the doubly nonlinear model is energetically convenient, since it is easy to see that, for every given choice of the parameters, the set of masses for which doubly nonlinear ground states exist contains that for which any of the models with a single nonlinearity admits ground states. 

\smallskip
Once the portrait for existence is clear, we can turn our attention to the singular limit of ground states on grids $\G_\varepsilon$ with $\varepsilon\to0$. Since this leads us to investigate the relation between problems on grids and in the plane, we first need to specify how to compare functions on $\G_\varepsilon$ with those on $\R^2$. To this end, we consider the following extension procedure. For every $\varepsilon>0$, we write
\[
\R^2=\bigcup_{(i,j)\in\Z^2}U_{ij}^\varepsilon\cup D_{ij}^\varepsilon\,,
\]
with
\begin{equation}
	\label{eq:UD}
	\begin{split}
		U_{ij}^\varepsilon&\,:=\left\{(x,y)\in\R^2\,:\,\varepsilon i\leq x\leq\varepsilon(i+1),\,x-\varepsilon i\leq y-\varepsilon j\leq\varepsilon\right\}\\
		D_{ij}^\varepsilon&\,:=\left\{(x,y)\in\R^2\,:\,\varepsilon i\leq x\leq\varepsilon(i+1),\,x-\varepsilon i\geq y-\varepsilon j\geq0\right\}
	\end{split}
\end{equation}
being the up-diagonal and down-diagonal triangles contained in the cell of $\G_\varepsilon$ with vertices $(\varepsilon i,\varepsilon j)$, $(\varepsilon(i+1),\varepsilon j)$, $(\varepsilon(i+1),\varepsilon(j+1))$, $(\varepsilon i, \varepsilon(j+1))$. By construction, for almost every $(x,y)\in\R^2$ there exists a unique couple $(i,j)\in\R^2$ such that $(x,y)$ belongs either to $U_{ij}^\varepsilon$ or to $D_{ij}^\varepsilon$. Hence, given $u:\G_\varepsilon\to\R$, we define its extension $\mathcal{A}u:\R^2\to\R$ inside each $U_{ij}^\varepsilon$, $D_{ij}^\varepsilon$ as the affine interpolation of the values of $u$ at the vertices of $U_{ij}^\varepsilon$, $D_{ij}^\varepsilon$, respectively. By definition, $\mathcal{A}u$ is piecewise affine and continuous on $\R^2$, and it coincides with $u$ at the vertices of $\G_\varepsilon$ (but not necessarily in the interior of its edges).

We can now state our main results on the singular limit of doubly nonlinear ground states. Note that, since such a limit requires to consider sequences of grids $\G_\varepsilon$ with varying $\varepsilon$, one must also specify how the set of nonlinear vertices $V_\varepsilon\subset\V_{\G_\varepsilon}$ changes with $\varepsilon$. In this sense, it is rather natural to consider sequences of $\Z$-periodic or $\Z^2$-periodic sets in $\G_\varepsilon$, as the periodicity is preserved as $\varepsilon\to0$. According to Definitions \ref{def:Z-per}--\ref{def:Z2-per}, this means that in the following we will always take $V_\varepsilon=\varepsilon V$, for a suitable $V\subseteq\V_{\G_1}$ that will possibly change case by case.

As it is reasonable to expect, the specific dislocation of the vertices of $V$ will affect the numerology in the next theorems. In particular, the choice of $\beta$ will be based on the following general fact: given a $\Z$-periodic or $\Z^2$-periodic set $V\subset\V_{\G_1}$, it is possible to identify a periodicity cell $V_0$ for $V$, i.e. a compact subset of $V$ such that the whole $V$ is given by the union of all translations of $V_0$ along the vector $\vec{v}$ of Definition \ref{def:Z-per}, if $V$ is $\Z$-periodic, or along the vectors $\vec{v}_1,\vec{v}_2$ of Definition \ref{def:Z2-per}, if $V$ is $\Z^2$-periodic. Moreover, to such $V_0$ one can naturally associate another compact set $Q_0\subset\G_1$, whose translations along the same vectors cover either the strip $\left\{P\in\G_1\,:\,|(P-P_0)\cdot \vec{v}^\perp|\leq r\right\}$ of Definition \ref{def:Z-per} when $V$ is $\Z$-periodic, or the whole $\G_1$ when $V$ is $\Z^2$-periodic. Even though the existence of such $V_0$ and $Q_0$ is heuristically evident by the very definitions of periodic subsets of $\V_{\G_1}$, for the sake of clarity we present the details of their construction in Remark \ref{rem:Per_dec} below.

Let us start considering the singular limit of doubly nonlinear ground states on grids with $\Z^2$-periodic nonlinear vertices. In this case, we prove that the limit problem in the plane is given by the energy functional
\begin{equation}
	\label{eq:ER2}
	E(u,\R^2):=\frac12\|\nabla u\|_{L^2(\R^2)}^2-\frac1p\|u\|_{L^p(\R^2)}^p-\frac1q\|u\|_{L^q(\R^2)}^q
\end{equation}
with two standard nonlinearities spread on the whole $\R^2$. As usual, let $\E_{\R^2}(\mu)$ be the corresponding ground state energy level in $H_\mu^1(\R^2)$, and recall that (by e.g. \cite{S20}) $\E_{\R^2}(\mu)$ is attained for every $\mu>0$ if and only if $2<p,q<4$, and that the associated ground states are solutions of 
\begin{equation}
	\label{NLS_Standard}
-\Delta u+\lambda u=|u|^{p-2}u+|u|^{q-2}u\qquad\text{in}\quad\R^2,
\end{equation}
where $\lambda\in\R$ denotes also in this context the Lagrange multiplier associated to the mass constraint. In this regime of nonlinearities, we have the following convergence result.
\begin{thm}
	\label{thm:limZ2}
	Let $p\in(2,4)$, $q\in(2,4)$, and $V\subseteq\V_{\G_1}$ be a given $\Z^2$-periodic set in $\G_1$. For every $\varepsilon>0$, let $V_\varepsilon:=\varepsilon V\subseteq\V_{\G_\varepsilon}$ be the $\Z^2$-periodic set in $\G_\varepsilon$ associated to $V$ in $\G_1$, and  
	\[
	\alpha = \frac12\,,\qquad\beta = \frac{\#\left(\V_{\G_1}\cap Q_0\right)}{\# V_0}\varepsilon\,,
	\]
	where $Q_0, V_0$ are the sets associated to $V$ as in Remark \ref{rem:Per_dec} below.
	Then, for every $\mu>0$, 
	\[
	\lim_{\varepsilon\to0}\varepsilon\E_{\G_\varepsilon}\left(\frac{2\mu}{\varepsilon}\right)=\E_{\R^2}(\mu)\,.
	\]
	Furthermore, for every positive ground state $u_\varepsilon$ of $E(\cdot,\G_\varepsilon)$ in $H_{\frac{2\mu}{\varepsilon}}^1(\G_\varepsilon)$ there exists $x_\varepsilon\in \R^2$ such that, up to subsequences,
	\[
	\mathcal{A}u_\varepsilon(\cdot-x_\varepsilon)\to\phi_\mu\quad\text{in }H^1(\R^2)\quad\text{as }\varepsilon\to0\,,
	\]
	where $\phi_\mu$ is a positive radially symmetric non-increasing ground state of $E(\cdot,\R^2)$ in $H_\mu^1(\R^2)$ and $\mathcal{A}$ is the extension operator introduced by \cite[Section 2]{D24}. 
\end{thm}
Observe that in the regime $2<p,q<4$ both $\E_{\R^2}(\mu)$ and $\E_{\G_\varepsilon}(\mu)$ are attained for every $\mu>0$, but this is not the only choice of the nonlinearity powers for which this is true on grids with $\Z^2$-periodic nonlinear vertices. Indeed, by Theorem \ref{thm:GSZ2} doubly nonlinear ground states on $\G_\varepsilon$ exist even when $p\in[4,6)$. Actually, adapting part of the proof of Theorem \ref{thm:limZ2} one can easily observe that the convergence of $\varepsilon\E_{\G_\varepsilon}(2\mu/\varepsilon)$ to $\E_{\R^2}(\mu)$ holds for every $p,q>2$ and $\mu>0$. Since $\E_{\R^2}(\mu)=-\infty$ whenever (the largest between) $p$ (and $q$) is greater than 4, this implies that the energy of doubly nonlinear ground states of $E(\cdot,\G_\varepsilon)$ in $H_{2\mu/\varepsilon}^1(\G_\varepsilon)$ diverges to $-\infty$ as $\varepsilon\to0$ in this case.

Let us now turn our attention to sequences of grids $\G_\varepsilon$ with $\Z$-periodic nonlinear vertices. Basing on Theorem \ref{thm:limZ2}, one expects again that in the limit for $\varepsilon\to0$ the nonlinearity concentrated on the set $V_\varepsilon$ converge to the $L^q$ norm on the subset of $\R^2$ somehow covered by ``the limit of $V_\varepsilon$''. According to this heuristics, there are two nontrivial possibilities for this limit subset, as it can be either a single line or a full strip in $\R^2$, in both cases parallel to the vector $\vec{v}$ of Definition \ref{def:Z-per}. Since by definition $V_\varepsilon=\varepsilon V$ with $V$ in $\G_1$, from the technical point of view one should recover the first limit problem when $V$ is a fixed set, whereas the second one should arise when $V$ is itself depending on $\varepsilon$ and its width in the direction orthogonal to $\vec{v}$ grows as $1/\varepsilon$ for $\varepsilon\to0$.

This is indeed the case. Let us introduce the energy functionals
\begin{equation}
	\label{eq:ER2_st}
	E_\theta(u,\R^2):=\frac12\|\nabla u\|_{L^2(\R^2)}^2-\frac1p\|u\|_{L^p(\R^2)}^p-\frac1q\|\tau_\theta u\|_{L^q(s_\theta)}^q
\end{equation}
\begin{equation}
	\label{eq:ER2_StR}
	E_{\theta,R}(u,\R^2):=\frac12\|\nabla u\|_{L^2(\R^2)}^2-\frac1p\|u\|_{L^p(\R^2)}^p-\frac1q\|u\|_{L^q(S_{\theta,R})}^q
\end{equation}
where, for every $\theta\in\left(-\frac\pi2,\frac\pi2\right]$ and $R>0$, we set $\vec{v}_\theta:=(\cos\theta, \sin\theta)$,
\begin{equation}
\label{eq:sSt}
s_\theta:=\left\{t\vec{v}_\theta\,:\,t\in\R\right\}\,,\qquad S_{\theta,R}:=\left\{P\in\R^2\,:\,\inf_{P_0\in s_\theta}|P-P_0|\leq R\right\}\,,
\end{equation}
and we let $\tau_\theta:H^1(\R^2)\to H^{1/2}(\R)$ be the trace operator on $s_\theta$, and denote by $\E_{\R^2,\theta}(\mu)$, $\E_{\R^2,\theta,R}(\mu)$ the corresponding ground state energy levels on $H_\mu^1(\R^2)$.

\begin{thm}
	\label{thm:limZ_theta}
	Let $p\in(2,4)$, $q\in(2,3)$,  and $V\subset\V_{\G_1}$ be a given $\Z$-periodic set in $\G_1$. For every $\varepsilon>0$, let $V_\varepsilon:=\varepsilon V\subseteq\V_{\G_\varepsilon}$ be the $\Z$-periodic set in $\G_\varepsilon$ associated to $V$ in $\G_1$, and 
	\[
	\alpha = \frac12\,,\qquad\beta = \frac{|\vec{v}|}{\#V_0}\,,
	\]
	where $\vec{v}:=(v_1,v_2)$ is the vector associated to $V$ as in Definition \ref{def:Z-per} and $V_0$ is the set associated to $V$ as in Remark \ref{rem:Per_dec} below. Let also
	\[
	\theta:=\begin{cases}
	\text{\normalfont arctan}\frac{v_2}{v_1} & \text{if }v_1\neq0\\
	\frac\pi2 & \text{if }v_1=0\,.
	\end{cases}
	\]
	Then, for every $\mu>0$, 
	\[
	\lim_{\varepsilon\to0}\varepsilon\E_{\G_\varepsilon}\left(\frac{2\mu}{\varepsilon}\right)=\E_{\R^2,\theta}(\mu)\,.
	\]
	Furthermore, for every positive ground state $u_\varepsilon$ of $E(\cdot,\G_\varepsilon)$ in $H_{\frac{2\mu}{\varepsilon}}^1(\G_\varepsilon)$ there exists $x_\varepsilon\in \R^2$ such that, up to subsequences,
	\[
	\mathcal{A}u_\varepsilon(\cdot-x_\varepsilon)\to\psi_\mu\quad\text{in }H^1(\R^2)\quad\text{as }\varepsilon\to0\,,
	\]
	where $\psi_\mu$ is a positive ground state of $E_\theta(\cdot,\R^2)$ in $H_\mu^1(\R^2)$. 
\end{thm}
\begin{thm}
	\label{thm:limZ_S}
	Let $p\in(2,4)$, $q\in(2,4)$, $R>0$ and $V\subset\V_{\G_1}$ be a given $\Z$-periodic set in $\G_1$. For every $\varepsilon>0$, set
	\[
	V_\varepsilon':=\bigcup_{\substack{i\in\Z \\ |i\varepsilon|\leq R}}\left(V+i\vec{v}^{\perp}\right)
	\]
	with $\vec{v}^\perp:=(-v_2,v_1)$, where $\vec{v}=(v_1,v_2)$ is the vector associated to $V$ as in Definition \ref{def:Z-per}. Let then $V_\varepsilon:=\varepsilon V_\varepsilon'$ be the $\Z$-periodic set in $\V_{\G_\varepsilon}$ associated to $V_\varepsilon'$ in $\G_1$, and
	\[
	\alpha = \frac12\,,\qquad\beta = \frac{\#\left(\V_{\G_1}\cap Q_0\right)}{\# V_0}\varepsilon\,,
	\]
	where $Q_0, V_0$ are the sets associated to $V$ as in Remark \ref{rem:Per_dec} below. Let also
	\[
	\theta:=\begin{cases}
		\text{\normalfont arctan}\frac{v_2}{v_1} & \text{if }v_1\neq0\\
		\frac\pi2 & \text{if }v_1=0\,.
	\end{cases}
	\]
	Then, for every $\mu>0$, 
	\[
	\lim_{\varepsilon\to0}\varepsilon\E_{\G_\varepsilon}\left(\frac{2\mu}{\varepsilon}\right)=\E_{\R^2,\theta,R}(\mu)\,.
	\]
	Furthermore, for every positive ground state $u_\varepsilon$ of $E(\cdot,\G_\varepsilon)$ in $H_{\frac{2\mu}{\varepsilon}}^1(\G_\varepsilon)$ there exists $x_\varepsilon\in \R^2$ such that, up to subsequences,
	\[
	\mathcal{A}u_\varepsilon(\cdot-x_\varepsilon)\to\varphi_\mu\quad\text{in }H^1(\R^2)\quad\text{as }\varepsilon\to0\,,
	\]
	where $\varphi_\mu$ is a positive ground state of $E_{\theta,R}(\cdot,\R^2)$ in $H_\mu^1(\R^2)$. 
\end{thm}
Theorems \ref{thm:limZ_theta}--\ref{thm:limZ_S} are particularly interesting as they show that models with nonlinear $\delta$-type vertex conditions on grids can be used to approximate nonlinear singular problems in Euclidean spaces. In particular, whereas the ground state problem $\E_{\R^2,\theta,R}(\mu)$ is similar to the ground state problem $\E_{\R^2}(\mu)$, as its ground states satisfy
\begin{equation}
	\label{NLS_Strip}
-\Delta u+\lambda u=|u|^{p-2}u+\chi_{S_{\theta,R}}|u|^{q-2}u\qquad\text{in}\quad\R^2
\end{equation}
($\chi_{S_{\theta,R}}$ denoting the characteristic function of $S_{\theta,R}$), the ground state problem $\E_{\R^2,\theta}$ is highly singular, as solutions to this problem satisfy the following singular NLS equation in $\R^2$
\begin{equation}
	\label{NLS_Singular}
\begin{cases}
	\displaystyle-\Delta u+\lambda u =|u|^{p-2}u & \text{in }\R^2\setminus s_\theta\\[.2cm]
	\displaystyle\frac{\partial u^+}{\partial\vec{v}_\theta^\perp}-\frac{\partial u^-}{\partial\vec{v}_\theta^\perp}=-|\tau_\theta u|^{q-2}\tau_\theta u & \text{in }s_\theta\,.
\end{cases}
\end{equation}
Existence results for ground states of $\E_{\R^2,\theta}$ and $\E_{\R^2,\theta,R}$ (and further details on \eqref{NLS_Singular}) are given in Section \ref{sec:R2} below. As in Theorem \ref{thm:limZ2}, also in Theorems \ref{thm:limZ_theta}--\ref{thm:limZ_S} the convergence of ground states is proved for all the values $p,q$ for which the limit problem admits ground states for every $\mu>0$. Even though the doubly nonlinear problem on grids admits existence of ground states for every mass in a larger regime of nonlinearities, it is again straightforward to show that the ground state energy level $\varepsilon\E_{\G_\varepsilon}(2\mu/\varepsilon)$ converges to that of the corresponding limit problem in $\R^2$ for every $p,q>2$, entailing again that the energy of ground states diverges to $-\infty$ whenever $p>4$ or $q>3$ in Theorem \ref{thm:limZ_theta} and $p>4$ in Theorem \ref{thm:limZ_S}.

To conclude, we point out that in principle it may be interesting to investigate the singular limit of ground states on $\G_\varepsilon$ even in the case of finitely many nonlinear vertices. For instance, one may be tempted to guess that, when $V$ is a fixed subset of $\V_{\G_1}$ with $\#V<\infty$, ground states of $E(\cdot,\G_\varepsilon)$  with $V_\varepsilon=\varepsilon V$ converge to ground states of some limit model in the plane with a nonlinearity concentrated at a single point. However, from the technical point of view this is expected to require sharp estimates on the $L^\infty$ norm of ground states on $\G_\varepsilon$ as $\varepsilon\to0$, in place of the analogous ones we derived for $L^p$ and $L^q$ norms in the proofs of Theorems \ref{thm:limZ2}--\ref{thm:limZ_theta}--\ref{thm:limZ_S}. Since such results on the $L^\infty$ norm are currently out of reach, at present we are not able to tackle this problem.

\smallskip
The remainder of the paper is organized as follows. Section \ref{sec:prel} discuss some preliminary results needed in the analysis. Section \ref{sec:exgsG} contains the discussion of our existence results for ground states on grids, proving Theorems \ref{thm:GScomp}--\ref{thm:GSZ}--\ref{thm:GSZ2}. Section \ref{sec:R2} proves existence and basic regularity for the ground state problems associated with $\E_{\R^2,\theta}$, $\E_{\R^2,\theta,R}$. The singular limit is addressed in the last two sections: in Section \ref{sec:limZ2} for $\Z^2$-periodic nonlinear vertices with the proof of Theorem \ref{thm:limZ2}, and in Section \ref{sec:limZ} for the $\Z$-periodic case with the proofs of Theorems \ref{thm:limZ_theta}--\ref{thm:limZ_S}.

\medskip
{\bf Notation.} In what follows, we will simply write $\|u\|_{p,\varepsilon}$ or $\|u\|_p$ for the $L^p$ norm of $u$ on the grid $\G_\varepsilon$, depending on whether it is important or not to underline the dependence on $\varepsilon$. The domain of integration will be written explicitly only in specific cases for which it is necessary.

\section{Preliminaries}
\label{sec:prel}

In this section we collect some preliminary facts and estimates that will be largely used in the rest of the paper.

Throughout, we will think of the two-dimensional grid $\G_\varepsilon=(\V_{\G_\varepsilon}, \mathbb{E}_{\G_\varepsilon})$ with edgelength $\varepsilon$ as the subset of $\R^2$ with vertices on $\varepsilon\Z^2$ and edges between every couple of vertices at distance $\varepsilon$ in $\R^2$, that is
\[
\vv\in\V_{\G_\varepsilon} \cong \varepsilon(i,j)\in\varepsilon\Z^2\subset\R^2
\]
and
\[
e\in \mathbb{E}_{\G_\varepsilon} \Longleftrightarrow e\cong\varepsilon(i,i+1)\times\left\{\varepsilon j\right\}\text{ or }e\cong\left\{\varepsilon i\right\}\times\varepsilon(j,j+1)\,,\text{ for some }(i,j)\in\Z^2\,.
\]
Sometimes, it will also be convenient to interpret $\G_\varepsilon$ as
\[
\G_\varepsilon=\bigg(\bigcup_{j\in\Z}H_{\varepsilon j}\bigg)\cup\bigg(\bigcup_{i\in\Z}V_{\varepsilon i}\bigg),
\]
where $H_{\varepsilon j}$, $V_{\varepsilon i}$ are the horizontal line $y=\varepsilon j$ and the vertical line $x=\varepsilon i$ in $\R^2$ respectively, or as
\begin{equation}
\label{eq:L}
\G_\varepsilon=\bigcup_{(i,j)\in\Z^2}L_{i,j}^\varepsilon\,,
\end{equation}
where $L_{i,j}^\varepsilon$ is the union of the vertex $\varepsilon(i,j)$ and of the edges $\varepsilon(i,i+1)\times\left\{\varepsilon j\right\}$, $\left\{\varepsilon i\right\}\times\varepsilon(j,j+1)$.

\begin{rem}
	\label{rem:Per_dec}
	In the following we will sometimes exploit specific periodic decompositions of $\G_\varepsilon$ induced by periodic subsets of its vertices. Indeed, taking for simplicity $\varepsilon=1$, it is easily seen that for every given set $V\subseteq\V_{\G_1}$ that is $\Z^2$-periodic according to Definition \ref{def:Z2-per}, i.e. 
	\[
	V=V+i\vec{v}_1+j\vec{v}_2\,,\qquad\forall (i,j)\in\Z^2
	\]
	for suitable linearly independent vectors $\vec{v}_1,\vec{v}_2\in\Z^2$, there exists a bounded set $Q_{0}\subset\G_1$ such that, setting $V_{0}:=Q_{0}\cap V$, there results
	\begin{equation}
		\label{eq-aim}
	V=\bigcup_{(i,j)\in\Z^2}(V_0+i\vec{v}_1+j\vec{v}_2),\qquad\G_1=\bigcup_{(i,j)\in\Z^2}(Q_0+i\vec{v}_1+j\vec{v}_2)
	\end{equation}
	and $\left(Q_0+i\vec{v}_1+j\vec{v}_2\right)\cap(Q_0+i'\vec{v}_1+j'\vec{v}_2)=\emptyset$ for every $(i,j)\neq(i',j')$. Such a set $Q_0$ can be constructed explicitly as follows. Writing $\vec{v}_1=(v_{1,x}, v_{1,y}),\vec{v}_2=(v_{2,x},v_{2,y})$, set $k:=|v_{2,x}v_{1,y}-v_{1,x}v_{2,y}|>0$ and 
	\[
	Q:=\bigcup_{-k\leq i,j<k}L_{i,j}^1
	\]
	with $L_{i,j}^1$ be as in \eqref{eq:L}. Since one can check that, by definition, the vectors $(k,0), (0,k)$ are integer combinations of $\vec{v}_1,\vec{v}_2$, it holds
	\begin{equation}
		\label{eq-first_cover}
		\G_1=\bigcup_{(i,j)\in\Z^2}(Q+i\vec{v}_1+j\vec{v}_2).
	\end{equation}
	Moreover, the set of vertices of $Q$ is $\Z^2\cap[-k,k)^2$. On this set, consider the equivalence
	\[
	\vv\sim\textsc{w} \Longleftrightarrow \vv = \textsc{w}+i\vec{v}_1+j\vec{v}_2\,,\text{ for some }(i,j)\in\Z^2
	\]
	and, for each equivalence class, pick the vertex closest to $(0,0)$ (note that such choice may not be unique). Let, then, $\left\{(i_n,j_n)\right\}_{n\in I}$ denote the sequence of such chosen vertices, with $I\subset\N$ finite since $\Z^2\cap[-k,k)^2$ is finite, and take
	\[
	Q_0:=\bigcup_{n\in I}L_{i_n,j_n}^1\,.
	\]
	By construction, $Q_0\subseteq Q$, $(Q_0+i\vec{v}_1+j\vec{v}_2)\cap(Q_0+i'\vec{v}_1+j'\vec{v}_2)=\emptyset$ for every $(i,j)\neq(i',j')$, and there exists $I_1,\,I_2\subset\N$ finite such that
	\[
	Q\subseteq \bigcup_{(i,j)\in I_1\times I_2}(Q_0+i\vec{v}_1+j\vec{v}_2).
	\]
	Hence, combining with \eqref{eq-first_cover}, one gets \eqref{eq-aim}.
	
	Observe that, given a $\Z$-periodic set $V\subset\V_{\G_\varepsilon}$ according to Definition \ref{def:Z-per}, arguing as above it is straightforward to construct $Q_0\subset\G_\varepsilon$ such that $(Q_0+i\vec{v})\cap(Q_0+i'\vec{v})=\emptyset$ for every $i,i'\in\Z$, $i\neq i'$, and 
	\[
	V=\bigcup_{i\in\Z}(V_0+i\vec{v}),\qquad \G_\varepsilon'=\bigcup_{i\in\Z}(Q_0+i\vec{v})\,,
	\]
	where $V_0:=Q_0\cap V$ and
	\begin{equation}
		\label{eq-Gprime}
	\G_\varepsilon':=\bigcup_{\varepsilon(i,j)\in J_\varepsilon}L^\varepsilon_{i,j}
	\end{equation}
	with
	\[
	J_\varepsilon:=\varepsilon\Z^2\cap\left\{P\in\R^2\,:\,|(P-P_0)\cdot \vec{v}^\perp|\leq r\right\}
	\]
	and $\vec{v}, P_0, r$ as in Definition \ref{def:Z-per}. Note also that the strip $\big\{P\in\G\,:\,|(P-P_0)\cdot \vec{v}^\perp|\leq r\big\}\subseteq\G_\varepsilon'$.
\end{rem}

A main tool in our analysis will be various Gagliardo-Nirenberg type inequalities on grids. The next lemma starts recalling some of them, that are by now well-known.
\begin{lem}\label{l.GN-eps.}
	Given $p\in(2,\infty]$ and $\varepsilon>0$, there results
	\begin{equation}\label{GN-eps.1}
		\|u\|_{p,\varepsilon}\lesssim \|u\|_{2,\varepsilon}^{\frac{1}{2}+\frac{1}{p}}\|u^\prime\|_{2,\varepsilon}^{\frac{1}{2}-\frac{1}{p}}\qquad\forall u\in H^1(\G_\varepsilon).
	\end{equation}
	Moreover, for every $p\in(2,\infty)$, there results
	\begin{equation}\label{GN-eps.2}
		\|u\|_{p,\varepsilon}\lesssim \varepsilon^{\frac{1}{2}-\frac{1}{p}}\|u\|_{2,\varepsilon}^{\frac{2}{p}}\|u^\prime\|_{2,\varepsilon}^{1-\frac{2}{p}}\qquad\forall u\in H^1(\G_\varepsilon),
	\end{equation}
	and, for every $p\in[4,6]$,
	\begin{equation}\label{GN-eps-interd.}
		\|u\|_{p,\varepsilon}\lesssim \varepsilon^\frac{6-p}{2p}\|u\|_{2,\varepsilon}^{1-\frac{2}{p}}\|u^\prime\|_{2,\varepsilon}^\frac{2}{p}\qquad\forall u\in H^1(\G_\varepsilon).
	\end{equation}
\end{lem}
\begin{proof}
	The case $\varepsilon=1$ has been proved in \cite[Theorem 2.1, Theorem 2.3, Corollary 2.4]{ADST19}, while the desired inequalities for $u\in H^1(\G_\varepsilon)$ with $\varepsilon\neq1$ follow by the corresponding ones on $\G_1$ applied to $v(x):=u(\varepsilon x)$.
\end{proof}

The next lemma provides a Gagliardo-Nirenberg type estimates for concentrated nonlinearities on $\Z^2$-periodic sets.
\begin{lem}\label{l.GN-grid-i}
	Let $q\geq2$, $\varepsilon>0$ and $V\subseteq \V_{\G_\varepsilon}$ be a $\Z^2$-periodic set according to Definition \ref{def:Z2-per}. Then
	\begin{equation}
		\label{eq-GNgrid2per}
	\left|\varepsilon\frac{2\#(\V_{\G_\varepsilon}\cap Q_0)}{\#V_0}\sum_{\vv\in  V}|u(\vv)|^q - \|u\|_{q,\varepsilon}^q\right|\lesssim \varepsilon\|u\|_{2(q-1),\varepsilon}^{q-1}\|u^\prime\|_{2,\varepsilon} \qquad \forall u\in H^1(\mathcal{G}_\varepsilon), 
	\end{equation}
	where $Q_0,V_0$ are the sets associated to $V$ as in Remark \ref{rem:Per_dec}.
\end{lem}
\begin{proof}
	Note first that by definition any $\Z^2$-periodic set $V\subseteq V_{\G_\varepsilon}$ satisfies $V=\varepsilon V'$ for some $\Z^2$-periodic set $V'\subseteq\V_{\G_1}$. As a consequence, if \eqref{eq-GNgrid2per} holds for $\varepsilon=1$, then it holds for $\varepsilon\neq0$ too. Indeed, since $w(x):=\varepsilon^{1/q}u(\varepsilon x)$ belongs to $H^1(\G_1)$ for every $u\in H^1(\G_\varepsilon)$, we have
	\[
	\begin{split}
	\left|\varepsilon\frac{2\#(\V_{\G_\varepsilon}\cap Q_0)}{\#V_0}\sum_{\vv\in V}|u(\vv)|^q-\|u\|_{q,\varepsilon}^q\right|&\,=\left|\frac{2\#(\V_{\G_1}\cap Q_0')}{\#V_0'}\sum_{\vv\in V'}|w(\vv)|^q-\|w\|_{q,1}^q\right|\\
	&\, \lesssim \|w\|_{2(q-1),1}^{q-1}\|w^\prime\|_{2,1}=\varepsilon\|u\|_{2(q-1),\varepsilon}^{q-1}\|u^\prime\|_{2,\varepsilon}.
	\end{split}
	\]
	
	We are thus left to prove \eqref{eq-GNgrid2per} for $\varepsilon=1$. Consider first the case $ V=\mathbb Z^2$.  Given $u\in H^1(\G_1)$, by \eqref{eq:L}, the fundamental theorem of calculus and H\"older inequality we have
	\begin{equation}
	\label{eq:estZ2}
	\begin{split}
	\left|2\sum_{\vv\in\mathbb Z^2} |u(\vv)|^q - \|u\|_{q,1}^q\right|&\,\le \sum_{(i,j)\in\mathbb Z^2} \left|2|u(i,j)|^q - \|u\|_{L^q(L_{i,j}^1)}^q\right| =\sum_{(i,j)\in\mathbb Z^2}\left|\int_{L_{i,j}^1}(|u(i,j)|^q-|u(x)|^q)\,dx\right|\\
	&\,\le \sum_{(i,j)\in\mathbb Z^2}\int_{L_{i,j}^1}\int_{L_{i,j}^1}|(|u(y)|^q)'|\,dy\,dx=2q\sum_{(i,j)\in\mathbb Z^2}\int_{L_{i,j}^1}|u(y)|^{q-1}|u'(y)|\,dy\\
	&\,=2q\int_{\G_1}|u(y)|^{q-1}|u'(y)|\,dy\lesssim \|u\|_{2(q-1),1}^{q-1}\|u'\|_{2,1}\,,
	\end{split}
	\end{equation}
	which proves \eqref{eq-GNgrid2per} for $V=\Z^2$, since in this case by Remark \ref{rem:Per_dec} one can take $Q_0=L_{0,0}^1$, so that $V_0=\left\{(0,0)\right\}$ and $\frac{2\#(\V_{\G_1}\cap Q_0)}{\#V_0}=2$.
	
	Consider now a general $\Z^2$-periodic set $V\subset \V_{\G_1}$, and let $Q_0, V_0$  be the corresponding sets as in Remark \ref{rem:Per_dec}.  Since we can write
	\[
		\begin{split}
		 \frac{\#(\mathbb Z^2\cap Q_0)}{\#V_0}\sum_{\vv\in V_0}|u(\vv)|^q&\,- \sum_{\vv\in\mathbb Z^2\cap Q_0}|u(\vv)|^q\\
		 &\,= \frac1{\#V_0}\left(\#(\mathbb Z^2\cap Q_0)\sum_{\vv\in V_0}|u(\vv)|^q- \#V_0\sum_{\vv\in\mathbb Z^2\cap Q_0}|u(\vv)|^q\right)\\
		 &\,=\frac1{\#V_0}\sum_{\ell=1}^m\left(|u(v_\ell)|^q-|u(w_\ell)|^q\right)
		 \end{split}
	\]
	with $m=\#(\mathbb Z^2\cap Q_0) \#V_0$ and $(v_\ell)_{\ell=1}^m,\,(w_\ell)_{\ell=1}^m$ two finite sequences of possibly non distinct vertices in $\mathbb Z^2\cap Q_0$, arguing as before we obtain
	\[
	\begin{split}
	  \left|\frac{\#(\mathbb Z^2\cap Q_0)}{\#V_0}\sum_{\vv\in V_0}|u(\vv)|^q- \sum_{\vv\in\mathbb Z^2\cap Q_0}|u(\vv)|^q\right|&\,\leq\sum_{\ell=1}^m\left||u(v_\ell)|^q-|u(w_\ell)|^q\right|\\
	  &\,\leq qm\int_{Q_0}|u(x)|^{q-1}|u^\prime(x)|dx, 
	\end{split}
	 \]
	where $L$ is the diameter of $Q_0$. Clearly, these computations do not change if we replace $Q_0,V_0$ by $Q_0+i\vec{v}_1+j\vec{v}_2, V_0+i\vec{v}_1+j\vec{v}_2$, for every $(i,j)\in\Z^2$. Therefore, since Remark \ref{rem:Per_dec} ensures that $(Q_0+i\vec{v}_1+j\vec{v}_2)_{(i,j)\in\Z^2}, (V_0+i\vec{v}_1+j\vec{v}_2)_{(i,j)\in\Z^2}$ are disjoint partitions of $\G_1, V$ respectively, we have
	\[
	\begin{split}
	&\left|\frac{\#(\V_{\G_1}\cap Q_0)}{\#V_0}\sum_{\vv\in V}|u(\vv)|^q-\sum_{\vv\in\Z^2}|u(\vv)|^q\right|\\
	&\qquad=\left|\sum_{(i,j)\in\mathbb Z^2}\left(\frac{\#(\V_{\G_1}\cap Q_0)}{\#V_0}\sum_{\vv\in V_0+i\vec{v}_1+j\vec{v}_2}|u(\vv)|^q-\sum_{\vv\in \Z^2\cap(Q_0+i\vec{v}_1+j\vec{v}_2)}|u(\vv)|^q\right)\right|\\
	&\qquad\lesssim\sum_{(i,j)\in\mathbb Z^2}\int_{Q_0+i\vec{v}_1+j\vec{v}_2}|u(x)|^{q-1}|u'(x)|\,dx=\int_{\G_1}|u(x)|^{q-1}|u'(x)|\,dx\leq\|u\|_{2(q-1),1}^{q-1}\|u'\|_{2,1}\,,
	\end{split}
	\]
	and, combining with \eqref{eq:estZ2}, we conclude.
\end{proof}

A similar result holds true also for $\Z$-periodic sets of vertices.
\begin{lem}
	\label{l.GN-grid.Zper.}
	Let $q\geq2$, $\varepsilon>0$ and $V\subset \V_{\G_\varepsilon}$ be a $\Z$-periodic set according to Definition \ref{def:Z-per}. Then
	\[
	\varepsilon\sum_{\vv\in V}|u(\vv)|^q\lesssim\|u\|_{L^q(\G_\varepsilon')}^q+ \varepsilon\|u\|_{L^{2(q-1)}(\G_\varepsilon')}^{q-1}\|u'\|_{L^2(\G_\varepsilon')} \qquad\forall u\in H^1(\mathcal{G}_\varepsilon)\,,
	\]
	with $\G_\varepsilon'$ defined by \eqref{eq-Gprime}.
\end{lem}
\begin{proof}
	Since every $\Z$-periodic set $V$ in $\G_\varepsilon$ can be written as $V=\varepsilon V'$ with $V'$ being $\Z$-periodic in $\G_1$, the result for $u\in H^1(\G_\varepsilon)$ directly follows by applying the lemma to $w(x)=\varepsilon^{1/q}u(\varepsilon x)\in H^1(\G_1)$. Hence, it is again enough to prove the claim when $\varepsilon=1$. In this case, the proof follows the same computations in \eqref{eq:estZ2} (but without a focus on the exact constants involved). 
\end{proof}
Since the previous lemma introduces a relation between $\Z$-periodic concentrated nonlinearities and Lebesgue norms on strip-like subsets of the grid, we conclude this section recalling the following estimate, a proof of which can be found in \cite[Lemma 2.3]{BDS23}.
\begin{lem}
\label{l.GN-grid.Zper.0}
	Let $\widetilde{\mathcal{G}}\subseteq\mathcal{G}_1$ be a subgraph with $\big|\widetilde{\mathcal{G}}\big|>0$ satisfying
 	\[
 	\min\left\{\sup_{j\in\mathbb Z}\#(\mathbb V_{\widetilde{\mathcal{G}}}\cap H_j),\,\sup_{j\in\mathbb Z}\#(\mathbb V_{\widetilde{\mathcal{G}}}\cap V_j)\right\}<+\infty, 
 	\]
	where $H_j, V_j$ are the straight lines $y=j$ and $x=j$, respectively. Then, for every $q>2$, there results
	\[
	\|u\|_{L^q(\widetilde{\mathcal{G}})}^q\lesssim \|u\|_{2,1}\|u^\prime\|_{2,1}^{q-1} \qquad \forall u\in H^1(\mathcal{G}_1).
	\]
\end{lem}

\section{Existence of ground states on grids: proof of Theorems \ref{thm:GScomp}--\ref{thm:GSZ}--\ref{thm:GSZ2}}
\label{sec:exgsG}
Here we prove our main existence/non-existence results for ground states on two-dimensional grids with finitely many (Theorem \ref{thm:GScomp}), $\Z$-periodic (Theorem \ref{thm:GSZ}) and $\Z^2$-periodic (Theorem \ref{thm:GSZ2}) concentrated nonlinearities. 

We give the details of the proofs in the case $\varepsilon=\alpha=\beta=1$, since different values of the parameters affect only the actual values of the thresholds $\overline\mu$ in Theorems \ref{thm:GScomp}--\ref{thm:GSZ}. Hence, all along this section the symbols $E(u,\G),\, \E_{\G}(\mu)$ will denote the quantities in \eqref{eq:EG}, \eqref{eq:levEG} with $\alpha=\beta=1$ on $\G=\G_1$, and $\|\cdot\|_r$ will denote the standard Lebesgue norm in $L^r(\G)$.

\begin{rem}
\label{rem:l.cont.}
Note that, given $p\in(2,6)$, $q\in(2,4)$ and $V\subseteq\V_\G$, the function $\mu\mapsto \E_\G(\mu)$ is non-positive and continuous on $\mu\in[0,\infty)$ . Indeed, in view of Lemma \ref{l.GN-eps.}, the fact that $\E_\G(\mu)\leq 0$ can be easily seen taking any sequence $(u_n)_n\subset H_\mu^1(\G)$ with $\|u_n'\|_2\to0$ as $n\to+\infty$. On the other hand, since 
\[
\mathcal{E}_{\G}(\mu)=\inf_{v\in H^1_\mu(\mathcal{G})}E(v,\G)=\inf_{u\in H^1_1(\mathcal{G})}E(\sqrt{\mu}u,\G)\,,
\]
and since for every $u\in H_1^1(\G)$ the quantity
\[ E(\sqrt{\mu}u,\G)=\frac{\mu}{2}\|u^\prime\|_2^2-\frac{\mu^{\frac{p}{2}}}{p}\|u\|_p^p-\frac{\mu^{\frac{q}{2}}}{q}\sum_{\vv\in V}|u(\vv)|^q
\]
is a concave function of $\mu\in[0,\infty)$, then $\mathcal{E}_{\G}(\mu)$ is concave too and, thus, continuous on $(0,+\infty)$. However, as $\mathcal{E}_{\G}(0)=0$ and $\E_\G(\mu)\leq 0$, for every $\mu\in[0,+\infty)$, it is straightforward that $\mathcal{E}_{\G}(\mu)$ is continuous on $[0,+\infty)$.
\end{rem}

The next lemma provides a sufficient condition for existence of ground states.

\begin{lem}
\label{l.ex.crit.}
Let $p\in(2,6)$, $q\in(2,4)$, $\mu>0$ and $V\subseteq\V_\G$ be either bounded, $\mathbb Z$-periodic or $\mathbb Z^2$-periodic. If
\begin{equation}
\label{eq:E<0}
\mathcal{E}_{\G}(\mu)<0,
\end{equation}
then a ground state of $E(\cdot,\G)$ in $H_\mu^1(\G)$ exists.
\end{lem}
\begin{proof}
	Let $(u_n)_n\subset H_\mu^1(\G)$ be such that $E(u_n,\G)\to\E_\G(\mu)$ as $n\to+\infty$. Since $p\in(2,6)$ and $q\in(2,4)$, by Lemma \ref{l.GN-eps.} it is easily seen that $(u_n)_n$ is bounded in $H^1(\G)$, so that up to subsequences $u_n\rightharpoonup u$ in $H^1(\G)$ and $u_n\to u$ in $L^\infty_{loc}(\G)$ as $n\to+\infty$. By weak lower semicontinuity, $0\leq\|u\|_2^2\le \mu$. Note that, if $\|u\|_2^2=\mu$, then $u\in H_\mu^1(\G)$ and the convergence of $u_n$ to $u$ is strong in $L^2(\G)$, and thus in $L^r(\G)$ for every $r\in[2,\infty]$ again by Lemma \ref{l.GN-eps.}. Moreover, $\displaystyle\sum_{\vv\in V}|u_n(\vv)|^q\to\sum_{\vv\in V}|u(\vv)|^q$ too. This is evident if $V$ is bounded, whereas it follows e.g. by Lemmas \ref{l.GN-grid-i}--\ref{l.GN-grid.Zper.} when $V$ is either $\Z^2$-periodic or $\Z$-periodic, respectively. Therefore, again by lower semicontinuity, we have
	\[
	\E_\G(\mu)=\lim_{n\to+\infty}E(u_n,\G)\geq E(u,\G)\geq \E_\G(\mu)\,,
	\]
	i.e. $u$ is a ground state. Hence, to complete the proof it is left to prove that $\|u\|_2^2\not\in[0,\mu)$.
	
	Assume, first, by contradiction that $\|u\|_2^2\in(0,\mu)$, and observe that
	\begin{equation}
	\label{eq:BL}
	E(u_n,\G)=E(u,\G) + E(u_n-u,\G) + o(1)\qquad\text{as}\quad n\to+\infty. 
	\end{equation}
	Indeed,	by direct computations we have
	\[
	\|u_n'-u'\|_2^2=\|u_n'\|_2^2 + \|u'\|_2^2 - 2\int_\G u_n'u'\,dx= \|u_n'\|_2^2-\|u'\|_2^2 + o(1)\,,
	\]
	whereas the relations
	\[
	\begin{split}
	\|u_n-u\|_p^p&\,=\|u_n\|_p^p-\|u\|_p^p + o(1)\\[.2cm]
	\sum_{\vv\in V}|u_n(\vv)-u(\vv)|^q&\,=\sum_{\vv\in V}|u_n(\vv)|^q - \sum_{\vv\in V}|u(\vv)|^q + o(1)
	\end{split}
	\]
	directly follow by Brezis-Lieb Lemma \cite{BL83}.  Now, since both $u\not\equiv0$ and $u_n-u\not\equiv0$ by assumption, making use of $p,q>2$ we obtain
	\[
	\begin{split}
	\mathcal{E}_\G(\mu)&\,\le E\left(\sqrt{\frac{\mu}{\|u\|_2^2}}u,\G\right)\\
	&\,= \frac{1}{2}\frac{\mu}{\|u\|_2^2}\|u'\|_2^2 - \frac{1}{p}\left(\frac{\mu}{\|u\|_2^2}\right)^{p/2}\|u\|_p^p - \frac{1}{q}\left(\frac{\mu}{\|u\|_2^2}\right)^{q/2}\sum_{\vv\in V}|u(\vv)|^q < \frac{\mu}{\|u\|_2^2}E(u,\G),
	\end{split}
	\]
	that is 
	\begin{equation}\label{proof.dic.--.1}
	E(u,\G)> \frac{\|u\|_2^2}{\mu}\mathcal{E}_{\G}(\mu)\,,
	\end{equation}
	and similarly
	\begin{equation}\label{proof.dic.--.2}
	\liminf_{n\to+\infty}E(u_n-u,\G)\ge \liminf_{n\to+\infty}\frac{\|u_n-u\|_2^2}{\mu}\mathcal{E}_\G(\mu)=\frac{\mu-\|u\|_2^2}{\mu}\E_\G(\mu)\,,
	\end{equation}
	where we used that $\|u_n-u\|_2^2= \|u_n\|_2^2-\|u\|_2^2 + o(1)$ as $n\to+\infty$. Hence, combining \eqref{eq:BL}, \eqref{proof.dic.--.1} and \eqref{proof.dic.--.2} yields
	\[
	\mathcal{E}_{\G}(\mu)=\lim_{n\to+\infty}E(u_n,\G)\geq \liminf_{n\to+\infty} E(u_n-u,\G) + E(u,\G)>\mathcal{E}_{\G}(\mu), 
	\]
	that is a contradiction.
	
	It thus remains to exclude that $\|u\|_2=0$, i.e. $u\equiv0$ on $\G$. Even though the basic idea is always the same, this step is slightly different depending on $V$ being bounded, $\Z$-periodic or $\Z^2$-periodic. For this reason, we now discuss independently each of these cases.
	
	\emph{Case (i): $V$ is $\Z^2$-periodic.} Since in this setting $E(\cdot,\G)$ is invariant under discrete translations according to the periodicity of $V$, with no loss of generality we can assume that each $u_n$ satisfies $\|u_n\|_\infty=\|u_n\|_{L^\infty(Q_0)}$, where $Q_0$ is the set associated to $V$ as in Remark \ref{rem:Per_dec}. As a consequence, $u_n\to0$ strongly in $L^\infty(\G)$ and thus also in $L^r(\G)$ for every $r>2$ as $n\to+\infty$. By Lemma \ref{l.GN-grid-i}, this implies $\displaystyle\sum_{\vv\in V}|u_n(\vv)|^q\to0$ too, in turn entailing
	\[
	\E_\G(\mu)=\lim_{n\to+\infty}E(u_n,\G)\geq\liminf_{n\to+\infty}\frac12\|u_n'\|_2^2\geq0\,.
	\]
	Since this contradicts \eqref{eq:E<0}, the lemma is proved when $V$ is $\Z^2$-periodic.
	
	\emph{Cases (ii)$\&$(iii): $V$ bounded or $\Z$-periodic.} Let then $V$ be either bounded or $\Z$-periodic. Observe that, since $u_n\rightharpoonup 0$ in $H^1(\G)$ as $n\to+\infty$ by assumption, in both cases we can further assume with no loss of generality that
	\begin{equation}
	\label{eq:Sqto0}
	\sum_{\vv\in V}|u_n(\vv)|^q\to0\qquad\text{as}\quad n\to+\infty\,.
	\end{equation}
	When $V$ is bounded, it is a direct consequence of the convergence in $L^\infty_{loc}(\G)$ of $u_n$ to 0. Conversely, when $V$ is $\Z$-periodic, since $E(\cdot,\G)$ is invariant under discrete translations according to the periodicity of $V$, we can assume without loss of generality that $\|u_n\|_{L^\infty(V)}=\|u_n\|_{L^\infty(V_0)}$ for every $n$, where $V_0$ is the set associated to $V$ as in Remark \ref{rem:Per_dec}. Then the local convergence to 0 of $u_n$ implies $\|u_n\|_{L^\infty(V)}\to0$ and, since $q>2$, estimating $|u_n(\vv)|^q\leq\|u_n\|_{L^\infty(V)}^{q-2}|u_n(\vv)|^2$ for every $\vv\in V$ and making use of Lemma \ref{l.GN-grid.Zper.} and of the boundeness of $(u_n)_n$ in $H^1(\G)$ to show that $\sum_{\vv\in V}|u_n(\vv)|^2$ is uniformly bounded, one recovers \eqref{eq:Sqto0}.
	
	By \eqref{eq:Sqto0}, we then have
	\begin{equation}
	\label{eq:E>Einf}
	\E_\G(\mu)=\lim_{n\to+\infty}E(u_n,\G)\geq\liminf_{n\to+\infty}\left(\frac12\|u_n'\|_2^2-\frac1p\|u_n\|_p^p\right)\geq\inf_{v\in H_\mu^1(\G)}\left(\frac12\|v'\|_2^2-\frac1p\|v\|_p^p\right).
	\end{equation}
	However, it has been proved in \cite[Theorems 1.1-1.2]{ADST19} that, for every $p\in(2,6)$, 
	\begin{equation}
		\label{eq-ADST}
	\inf_{v\in H_\mu^1(\G)}\left(\frac12\|v'\|_2^2-\frac1p\|v\|_p^p\right)\leq0\qquad\forall \mu>0
	\end{equation}
	and, for those values of $\mu>0$ for which it is strictly negative, there exists $u\in H_\mu^1(\G)$ such that
	\[
	\inf_{v\in H_\mu^1(\G)}\left(\frac12\|v'\|_2^2-\frac1p\|v\|_p^p\right)=\frac12\|u'\|_2^2-\frac1p\|u\|_p^p\,.
	\]
	Observe that, once any of such $u$ exists, it is easy to see that (up to a change of sign) $u>0$ on $\G$, so that by \eqref{eq:E>Einf}
	\[
	\E_\G(\mu)\geq\inf_{v\in H_\mu^1(\G)}\left(\frac12\|v'\|_2^2-\frac1p\|v\|_p^p\right)>\frac12\|u'\|_2^2-\frac1p\|u\|_p^p-\frac1q\sum_{\vv\in V}|u(\vv)|^q=E(u,\G)\geq\E_\G(\mu),
	\]
	which is a contradiction. On the other hand, whenever the equality holds in \eqref{eq-ADST}, one gets $\E_\G(\mu)\geq0$, which contradicts \eqref{eq:E<0} and concludes the proof.
\end{proof}

Now we can prove the main results on the existence of ground states on grids.

\begin{proof}[Proof of Theorem \ref{thm:GSZ2}]
	Let $p\in(2,6)$, $q\in(2,4)$ and $V\subseteq\V_\G$ be $\Z^2$-periodic. By \cite[Theorem 1.8]{BDS23}, for every $\mu>0$ there exists $u\in H_\mu^1(\G)$ such that
	\[
	\frac12\|u'\|_2^2-\frac1q\sum_{\vv\in V}|u(\vv)|^q<0\,.
	\]
	Since this automatically implies that $\E_\G(\mu)\leq E(u,\G)<0$, we conclude with a direct application of Lemma \ref{l.ex.crit.}.
\end{proof}

\begin{proof}[Proof of Theorem \ref{thm:GScomp}]
	Let $p\in(2,6)$, $q\in(2,4)$ and $V\subseteq\V_\G$ be bounded. By Lemma \ref{l.ex.crit.}, to show that $\E_\G(\mu)$ is attained it is enough to show that it is strictly negative. Set then 
	\[
	\overline\mu:=\inf\left\{\mu>0\,:\,\E_\G(\mu)<0\right\}\,.
	\]
	Note that, by \cite[Theorem 1.3]{BDS23}, if $\mu$ is sufficiently large there exists $u\in H_\mu^1(\G)$ such that 
	\[
	\frac12\|u'\|_2^2-\frac1q\sum_{\vv\in  V}|u(\vv)|^q<0\,,
	\]
	in turn yielding $\E_\G(\mu)<0$ for $\mu$ large enough. Hence, $\overline\mu<+\infty$. Moreover, if $\mu>0$ is such that $\E_\G(\mu)<0$, then by Lemma \ref{l.ex.crit.} there exists $u\in H_\mu^1(\G)$ such that $E(u,\G)=\E_\G(\mu)$. For every $\mu'>\mu$ we then have
	\begin{equation}
	\label{eq:scalmu}
	\E_\G(\mu')\leq E\left(\sqrt{\frac{\mu'}{\mu}}u,\G\right)=\frac12\frac{\mu'}{\mu}\|u'\|_2^2-\frac1p\left(\frac{\mu'}{\mu}\right)^\frac p2-\frac1q\left(\frac{\mu'}{\mu}\right)^\frac q2\sum_{\vv\in V}|u(\vv)|^q\leq\frac{\mu'}{\mu}E(u,\G)<0\,.
	\end{equation}
	This shows that $\E_\G(\mu)<0$ and, by Lemma \ref{l.ex.crit.} again, that ground states exist for every $\mu>\overline\mu$. Conversely, by definition of $\overline\mu$ and Remark \ref{rem:l.cont.} we immediately have that $\E_\G(\mu)=0$ for every $\mu\in[0,\overline\mu]$. Furthermore, if $\overline\mu>0$, then $\E_\G(\mu)$ is not attained whenever $\mu<\overline\mu$. Indeed, if by contradiction there exists $\mu'\in(0,\overline\mu)$ such that $\E_\G(\mu')$ is attained, then there exists $u\in H_{\mu'}^1(\G)$ such that $E(u,\G)=\E_\G(\mu')=0$. Taking then $\mu''\in(\mu',\overline\mu)$ and repeating the computations in \eqref{eq:scalmu} would yield $\E_\G(\mu'')<0$, which is impossible by definition of $\overline\mu$. 
	
	When $p\in(2,4)$ \cite[Theorem 1.1]{ADST19} claims that, for every $\mu>0$ there exists $u\in H_\mu^1(\G)$ such that
	\[
	\frac12\|u'\|_2^2-\frac1p\|u\|_p^p<0.
	\]
	Since this automatically implies that $\E_\G(\mu)\leq E(u,\G)<0$, we conclude once again by Lemma \ref{l.ex.crit.}.
	
	Therefore, to conclude the proof of Theorem \ref{thm:GScomp} it suffices to show that $\overline\mu>0$  whenever $p\in[4,6)]$. To this end, note that \eqref{GN-eps-interd.} entails
	\[
	E(u,\G) \geq\left(\frac{1}{2} - C_p\mu^\frac{p-2}{2}\right)\|u'\|_2^2 - \frac{1}{q}\sum_{\vv\in V}|u(\vv)|^q \ge  \left(\frac{1}{2} - C_p\mu^\frac{p-2}{2}\right)\|u'\|_2^2 - \frac{\#V}{q}\max_{\vv\in V}|u(\vv)|^q
	\]
	for every $u\in H^1(\G)$ and for a suitable constant $C_p>0$ depending only on $p$. When $\mu$ is sufficiently small, since $u$ cannot be constant, this gives
	\[
	E(u,\G) > \frac14\|u'\|_2^2-\frac{\#V}{q}\max_{\vv\in V}|u(\vv)|^q\geq\frac12\inf_{v\in H_\mu^1(\G)}\left(\frac12\|v'\|_2^2-\frac{2\#V}{q}|v(\overline{\vv})|^q\right)=0,
	\]
	with $\overline{\vv}:=\arg\max_{\vv\in V}|u(\vv)|^q$, where the last inequality is a consequence of \cite[Theorem 1.3]{BDS23} (see Section 5 therein for technical details). Hence, $E(u,\G)>0$ for every $u\in H_\mu^1(\G)$ with $\mu$ small enough, whence $\overline\mu>0$.
\end{proof}

\begin{proof}[Proof of Theorem \ref{thm:GSZ}]
	The line of the proof is almost identical to that of Theorem \ref{thm:GScomp}. Letting $p\in(2,6)$, $q\in(2,4)$ and $V\subset\V_\G$ be $\Z$-periodic, define again 
	\[
	\overline\mu:=\inf\left\{\mu>0\,:\,\E_\G(\mu)<0\right\}.
	\]
	Arguing exactly as before and recalling \cite[Theorem 1.7]{BDS23} when required, we obtain again that $\overline\mu<+\infty$ and that $\E_\G(\mu)$ is attained for every $\mu>\overline\mu$, whereas it is identically zero and not attained for $\mu<\overline\mu$, provided $\overline\mu>0$. Moreover, since in \cite[Theorem 1.7]{BDS23} it has been shown that for every $q\in(2,3)$ and every $\mu>0$ there exists $u\in H_\mu^1(\G)$ such that
	\[
	\frac12\|u'\|_2^2-\frac1q\sum_{\vv\in V}|u(\vv)|^q<0\,,
	\]
	this gives $\E_\G(\mu)\leq E(u,\G)<0$, that is $\overline\mu=0$ whenever $q\in(2,3)$, whence ground states exist for every $\mu>0$ (again by Lemma \ref{l.ex.crit.}). Analogously, since \cite[Theorem 1.1]{ADST19} ensures that, for every $p\in(2,4)$ and every $\mu>0$ there exists $u\in H_\mu^1(\G)$ such that $u>0$ on $\G$ and
	\[
	\frac12\|u'\|_2^2-\frac1p\|u\|_p^p<0\,,
	\]
	this entails again $\E_\G(\mu)\leq E(u,\G)<0$ for every $\mu>0$, leading again to the same result.
	
	To conclude, let then $p\in[4,6)$ and $q\in[3,4)$ and let us show that $\overline\mu>0$. To do this, it is enough to prove that $E(u,\G)>0$ for every $u\in H_\mu^1(\G)$ as soon as $\mu$ is sufficiently small.  
	
	Consider first the case $q=3$. Given $u\in H_\mu^1(\G)$, by Lemma \ref{l.GN-grid.Zper.}  and Young's inequality we have
	$$ 
	\frac1q\sum_{\vv\in V}|u(\vv)|^3\le C\left(\|u\|_{L^3(\mathcal{G}')}^3 + \|u\|_{L^4(\mathcal{G}^\prime)}^{2}\|u'\|_{L^2(\mathcal{G}^\prime)}\right)\le C\|u\|_{L^3(\mathcal{G}^\prime)}^3 +  \frac{1}{4}\|u'\|_{L^2(\mathcal{G}^\prime)}^2 + M\|u\|_{L^4(\mathcal{G}^\prime)}^{4}\,,
	$$
	where $\G'$ is the set associated to $V$ as in Lemma \ref{l.GN-grid.Zper.}, for suitable constants $C,M>0$ depending only on $q$ and $V$.
	Moreover, by Lemma \ref{l.GN-grid.Zper.0} and \eqref{GN-eps-interd.},
	$$
	 \|u\|_{L^3(\mathcal{G}')}^3\lesssim \mu^{1/2}\|u'\|_{2}^2,\qquad  \|u\|_{p}^p\lesssim \mu^\frac{p-2}{2}\|u'\|_{2}^2, \qquad \|u\|_{L^4(\mathcal{G}^\prime)}^4\leq\|u\|_{4}^4\lesssim \mu\|u'\|_{2}^2.
	$$
	Therefore, for $\mu>0$ sufficiently small we obtain (since $p\geq4$) 
	$$ 
	E(u,\G)\ge \left(\frac{1}{4}-C\mu^{1/2}\right)\|u'\|_{2,1}^2>0 \qquad \forall u\in H^1_\mu(\mathcal{G})\,, $$
	showing that $\overline\mu>0$ for every $p\in[4,6)$ and $q=3$.
	
	To recover the same result for every $q\in(3,4)$, assume now by contradiction that there exists $q\in(3,4)$ such that $\overline\mu=0$. By definition, this implies
	$$ 
	\mathcal{E}_\G(\mu)<0 \qquad \forall \mu>0,$$
	so that, by Lemma \ref{l.ex.crit.}, there exists $u_\mu\in H_\mu^1(\G)$ with $\mathcal{E}_\G(\mu)=E(u_\mu,\G)$ for every $\mu>0$. Since $\E_\G(\mu)$ is uniformly bounded in a neighbourhood of the origin, $(u_\mu)_\mu$ is uniformly bounded in $H^1(\mathcal{G})$ for $\mu>0$ sufficiently small. In particular, by Lemma \ref{l.GN-eps.}
	$$
	\|u_\mu\|_\infty\to 0\qquad \text{as}\quad\mu\to 0,
	$$
	and thus
	$$ 
	\frac1q\sum_{\vv\in V}|u_\mu(\vv)|^q\le \frac{1}{3}\sum_{\vv\in V}|u_\mu(\vv)|^3,
	$$
	yielding in turn 
	\[
	\frac12\|u_\mu'\|_2^2-\frac1p\|u_\mu\|_p^p-\frac13\sum_{\vv\in V}|u_\mu(\vv)|^3\leq \frac12\|u_\mu'\|_2^2-\frac1p\|u_\mu\|_p^p-\frac1q\sum_{\vv\in V}|u_\mu(\vv)|^q<0
	\]
	as soon as $\mu$ is small enough. However, this is prevented by the fact that $\overline\mu>0$ when $p\in[4,6)$ and $q=3$, proved before.
\end{proof}

\section{Ground states in $\R^2$ with standard and concentrated nonlinearities}
\label{sec:R2}
In this section we discuss existence results and basic properties for the doubly nonlinear problems in $\R^2$ that provide the limit models for ground states of $E(\cdot,\G_\varepsilon)$ on sequences of grids $\G_\varepsilon$ with vanishing edgelength.

First, it is now well-known (see e.g. the discussion in \cite{S20}) that, whenever $p\in(2,4)$ and $q\in(2,4)$,
\[
\E_{\R^2}(\mu):=\inf_{u\in H_\mu^1(\R^2)}E(u,\R^2)<0,
\]
$E(\cdot,\R^2)$ being the energy functional introduced by \eqref{eq:ER2}, and that ground states at mass $\mu$, denoted by $\phi_\mu$, exist for every $\mu>0$ and satisfy $\displaystyle \lim_{|x|\to+\infty}\phi_\mu(x)=0$ and, both in $L^2(\R)$ and in the classical sense, equation \eqref{NLS_Standard}. Furthermore, with no loss of generality, $\phi_\mu$ can be taken positive and attaining its $L^\infty$ norm at the origin, and standard rearrangement arguments readily show that $\phi_\mu$ is radially non-increasing on $\R^2$. The same result holds in the case of a single standard nonlinearity 
\begin{equation}
\label{eq:Ebar}
\overline{E}_p(u,\R^2):=\frac12\|\nabla u\|_{L^2(\R^2)}^2-\frac1p\|u\|_{L^p(\R^2)}^p
\end{equation}
whenever $p\in(2,4)$. Observe also that, by definition,
\begin{equation}
	\label{eq:2<1}
	\E_{\R^2}(\mu)<\min\left\{\inf_{v\in H_\mu^1(\R^2)}\overline{E}_p(u,\R^2),\,\inf_{v\in H_\mu^1(\R^2)}\overline{E}_q(u,\R^2)\right\}\,.
\end{equation}

We now want to derive analogous results for ground states of the functionals $E_\theta(\cdot,\R^2)$ and $E_{\theta,R}(\cdot,\R^2)$ introduced by \eqref{eq:ER2_st} and  \eqref{eq:ER2_StR}, respectively. Throughout, we use the following notation: $\vec{v}_\theta=(\cos\theta,\sin\theta)$, $s_\theta$ and $S_{\theta,R}$ are the sets defined in \eqref{eq:sSt} and $\tau_\theta: H^1(\R^2)\to H^\frac12(s_\theta)$ is the trace operator associated to $s_\theta$. Recall also that, by classical trace theory, such operator is bounded and surjective. Finally, for the sake of simplicity we will identify $s_\theta$ with $\R$ whenever this does not give rise to misunderstandings.

The next theorem summarizes our main existence results in this direction.
\begin{thm}
	\label{thm:exGS_theta}
	Let $\theta\in\left(-\frac\pi2,\frac\pi2\right]$ and $R>0$. Then
	\begin{itemize}
		\item[(i)] if $p\in(2,4)$ and $q\in(2,3)$, $\E_{\R^2,\theta}(\mu)$ is attained for every $\mu>0$;
		
		\item[(ii)] if $p\in(2,4)$ and $q\in(2,4)$, $\E_{\R^2,\theta,R}(\mu)$ is attained for every $\mu>0$.
	\end{itemize}
\end{thm}
In the proof of Theorem \ref{thm:exGS_theta}, we will use the following well-known two-dimensional Gagliardo-Nirenberg inequality
	\begin{equation}
	\label{eq:GNR2_clas}
	\|u\|_{L^p(\R^2)}^p\lesssim \|u\|_{L^2(\R^2)}^2\|\nabla u\|_{L^2(\R^2)}^{p-2}\qquad\forall u\in H^1(\R^2)
	\end{equation}
holding for every $p>2$ (see e.g. \cite{Leoni}), and, for every $q>2$,
\begin{equation}
\label{eq:GNR2delta}
		\|\tau_\theta u\|_{L^q(\R)}^q\lesssim\|u\|_{L^2(\R^2)}\|\nabla u\|_{L^2(\R^2)}^{q-1}\qquad\forall u\in H^1(\R^2)\,.
\end{equation}
The validity of the latter estimate can be easily seen when $u\in C_0^\infty(\R^2)$, since in this case
\[
\|\tau_\theta u\|_{L^q(\R)}^q\lesssim\int_{\R^2}|u|^{q-1}|\nabla u|\,dxdy\lesssim\|u\|_{L^{2(q-1)}(\R^2)}^{q-1}\|\nabla u\|_{L^2(\R^2)}\,,
\]
that together with \eqref{eq:GNR2_clas} gives \eqref{eq:GNR2delta}. The extension to a general $u\in H^1(\R^2)$ then follows by density, since by standard one-dimensional Sobolev embeddings and the boundedness of $\tau_\theta:H^s(\R^2)\to H^{s-\frac12}(\R)$ for every $s\in\left(\frac12,\frac32\right)$, it holds
\begin{equation}
	\label{eq:embtrace}
\|\tau_\theta u\|_{L^q(\R)}\lesssim\|\tau_\theta u\|_{H^{\frac12-\frac1q}(\R)}\lesssim\|u\|_{H^{1-\frac1q}(\R^2)}\,,
\end{equation}
that combined with the following Gagliardo-Nirenberg interpolation inequality (which can be easily checked, for instance, via Fourier transform)
\[
\|u\|_{H^{1-\frac1q}(\R^2)}\lesssim\|u\|_{H^1(\R^2)}^{1-\frac1q}\|u\|_{L^2(\R^2)}^\frac1q
\]
ensures that $\|\tau_\theta u_n-\tau_\theta u\|_{L^q(\R)}\to0$ if $u_n\to u$ in $H^1(\R^2)$ as $n\to+\infty$.
\begin{proof}[Proof of Theorem \ref{thm:exGS_theta}]
Let us prove the result for $E_\theta(\cdot,\R^2)$ first. Note that, by \eqref{eq:GNR2_clas} and \eqref{eq:GNR2delta}, 
\begin{equation}
\label{eq:coercE}
E_\theta(u,\R^2)>\frac12\|\nabla u\|_2^2-\frac{C_p}{p}\mu\|\nabla u\|_2^{p-2}-\frac{C_q}{q}\sqrt{\mu}\|\nabla u\|_{2}^{q-1}\qquad\forall u\in H^1(\R^2)\,,
\end{equation}
for suitable constants $C_p,C_q>0$ depending only on $p$ and $q$. Hence, $\E_{\R^2,\theta}(\mu)>-\infty$ for every $\mu>0$ whenever $p\in(2,4)$ and $q\in(2,3)$. Furthermore, since (as we recalled at the beginning of the section) for every $p\in(2,4)$ and $\mu>0$ there exists a positive $u\in H_\mu^1(\R^2)$ such that 
\[
\frac12\|\nabla u\|_{L^2(\R^2)}^2-\frac1p\|u\|_{L^p(\R^2)}^p=\inf_{v\in H_\mu^1(\R^2)}\overline{E}_p(v,\R^2)<0\,,
\]
there results that
\begin{equation}
\label{eq:stima}
\E_{\R^2,\theta}(\mu)<\inf_{v\in H_\mu^1(\R^2)}\overline{E}_p(v,\R^2)<0
\end{equation}
(where $\overline{E}_p(\cdot,\R^2)$ is defined by \eqref{eq:Ebar}).

Let then $(u_n)_n\subset H_\mu^1(\R^2)$ be such that $E_{\theta}(u_n,\R^2)\to\E_{\R^2,\theta}(\mu)$ as $n\to+\infty$. Exploiting the invariance of $E_\theta(\cdot,\R^2)$ by translations along vectors parallel to $\vec{v}_\theta$, we can further assume that, for every $n$,
\begin{equation}
	\label{eq:invar}
	\|\tau_\theta u_n\|_{L^2(0,1)}=\max_{j\in\N}\|\tau_\theta u_n\|_{L^2(j,j+1)}\,.
\end{equation}
 By \eqref{eq:coercE}, $(u_n)_n$ is bounded in $H^1(\R^2)$ and (up to subsequences) 
 \[
 \begin{split}
 	&u_n\rightharpoonup u\:\text{ in }H^1(\R^2)\,\qquad\qquad u_n\to u\:\text{ in }L_{loc}^r(\R^2)\,,\quad\forall r\geq2,\\
 	&\tau_\theta u_n\rightharpoonup \tau_\theta u\:\text{ in }H^\frac12(\R)\,\qquad \tau_\theta u_n\to \tau_\theta u\:\text{ in }L_{loc}^r(\R)\,,\quad\forall r\geq2\,.
 \end{split}
 \]
 By weak lower semicontinuity, $\|u\|_{L^2(\R^2)}^2\in[0,\mu]$. Moreover, if $\|u\|_{L^2(\R^2)}^2=\mu$, then $u\in H_\mu^1(\R^2)$ and, by \eqref{eq:GNR2_clas} and \eqref{eq:GNR2delta}, the convergence of $u_n$ to $u$ and of $\tau_\theta u_n$ to $\tau_\theta u$ will be strong in $L^p(\R^2)$ and in $L^q(\R)$, respectively, so that by lower semicontinuity again one concludes that $E_\theta(u,\R^2)=\E_{\R^2,\theta}(\mu)$, i.e. $u$ is a ground state at mass $\mu$. 
 
 To complete the proof for $E_\theta(\cdot,\R^2)$ it is then enough to exclude that $\|u\|_2^2\in[0,\mu)$.  Since the fact that $\|u\|_2^2\not\in(0,\mu)$ can be proved (in a very classical way) as in the first part of the proof of Lemma \ref{l.ex.crit.} above, assume by contradiction $\|u\|_2=0$, i.e. $u\equiv 0$ on $\R^2$. Recalling the standard Gagliardo-Nirenberg inequality 
 \[
 \|v\|_{H^{\frac12 -\frac1q}(I)}^q\lesssim \|v\|_{H^\frac12(I)}^{q-2}\|v\|_{L^2(I)}^{2}\qquad\forall u\in H^\frac12(I)
 \]
 holding on every interval $I\subset\R$ (see, e.g., \cite{BM18}), by \eqref{eq:embtrace} and \eqref{eq:invar} and the fact that the convergence of $\tau_\theta u_n$ to $0$ is locally strong in $L^2(\R)$, we obtain
 \[
 \begin{split}
 	\|\tau_\theta u_n\|_{L^q(\R)}^q=&\,\sum_{j\in\N}\|\tau_\theta u_n\|_{L^q(j,j+1)}^q\lesssim \sum_{j\in\N} \|\tau_\theta u_n\|_{H^\frac12(j,j+1)}^{q-2}\|\tau_\theta u_n\|_{L^2(j,j+1)}^{2}\\
 	\lesssim&\,\|\tau_\theta u_n\|_{L^2(0,1)}^{2} \sum_{j\in\N} \|\tau_\theta u_n\|_{H^\frac12(j,j+1)}^{q-2}=\|\tau_\theta u_n\|_{L^2(0,1)}^{2}\|\tau_\theta u_n\|_{H^\frac12(\R)}^{q-2}\longrightarrow0
 \end{split}
 \]
 as $n\to+\infty$. Hence, we have
 \[
 \begin{split}
 \E_{\R^2,\theta}(\mu)=\lim_{n\to+\infty}E_\theta(u_n,\R^2)\geq&\,\liminf_{n\to+\infty}\left(\frac12\|\nabla u_n\|_{L^2(\R^2)}^2-\frac1p\|u_n\|_{L^p(\R^2)}^p\right)\\[.2cm]
 =&\,\liminf_{n\to+\infty}\overline{E}_p(u_n,\R^2)\geq\inf_{v\in H_\mu^1(\R^2)}\overline{E}_p(v,\R^2)\,,
 \end{split}
 \]
 contradicting \eqref{eq:stima} and thus proving the claim for $E_\theta(\cdot,\R^2)$.
 
 The proof of the analogous result for $E_{\theta,R}(\cdot,\R^2)$ is almost identical. The only differences are:
 \begin{itemize}[label=$\star$]
 	\item the proof of the lower boundedness of $E_{\theta,\R}(\cdot, \R^2)$, which we obtain here for every $p\in(2,4)$ and $q\in(2,4)$ using \eqref{eq:GNR2_clas} (together with the trivial estimate $\|u\|_{L^q(S_{\theta,\R})}^q\leq\|u\|_{L^q(\R^2)}$) in place of \eqref{eq:GNR2delta};
 	
 	\item the choice of a minimizing sequence satisfying $\displaystyle\|u_n\|_{L^2(R_0)}=\max_{j\in\N}\|u_n\|_{L^2(R_j)}$ instead of \eqref{eq:invar}, where $(R_j)_{j\in\N}$ is a disjoint partition of $S_{\theta,R}$ in identical parallelograms with two edges parallel to $s_\theta$ of length 1;
 	
 	\item the proof of the fact that $\|u_n\|_{L^q(S_{\theta,R})}\to0$ whenever $u_n\rightharpoonup0$ as $n\to+\infty$, which for a general $q\in(2,4)$ follows here by interpolating between the boundedness of $(u_n)_n$ in $H^1(\R^2)$ and the estimate
 	\[
 	\begin{split}
 	\|u_n\|_{L^q(S_{\theta,\R})}^q=\sum_{j\in\N}\|u_n\|_{L^q(R_j)}^q\lesssim&\,\sum_{j\in\N}\|u_n\|_{L^2(R_j)}^2\|u_n\|_{H^1(R_j)}^{q-2}\\[.2cm]
 	\leq &\,\|u_n\|_{L^2(R_0)}^2\|u_n\|_{H^1(\R^2)}^{q-2}\to0
 	\end{split}
 	\]
 	as $n\to+\infty$ if $u_n\to0$ in $L^2_{loc}(\R^2)$ (the first inequality in the previous chain being justified by standard Gagliardo-Nirenberg inequalities on bounded sets of $\R^2$).
 \end{itemize}   
\end{proof}

\begin{rem}
	\label{rem:unbound}
	Observe that the regimes of nonlinearities given in Theorem \ref{thm:exGS_theta} are the only ones for which $E_\theta(\cdot,\R^2)$ and $E_{\theta,R}(\cdot,\R^2)$ admit ground states for every value of $\mu>0$. Indeed, since for every $p>2$ we always have
	\[
	\E_{\R^2,\theta}(\mu)\leq\inf_{v\in H_\mu^1(\R^2)}\overline{E}_p(v,\R^2), \qquad\E_{\R^2,\theta,R}(\mu)\leq\inf_{v\in H_\mu^1(\R^2)}\overline{E}_p(v,\R^2), 
	\]
	with $\overline E_p(\cdot,\R^2)$ as in \eqref{eq:Ebar}, it follows immediately that $\E_{\R^2,\theta}(\mu)=\E_{\R^2,\theta,R}(\mu)=-\infty$ for every $\mu>0$ when $p>4$ and for sufficiently large $\mu$ when $p=4$ (since it is well-known this is true for $\displaystyle\inf_{v\in H_\mu^1(\R^2)}\overline{E}(v,\R^2)$). Analogous arguments show also that $\E_{\R^2,\theta,R}(\mu)=-\infty$ for every $\mu>0$ when $q>4$ and for sufficiently large $\mu$ when $q=4$, even if $p\in(2,4)$.  As for $\E_{\R^2,\theta}$ when $q\geq3$, note that, for every $u\in H_\mu^1(\R^2)$ and every $\lambda>0$, the function $u_\lambda(x):=\lambda u(\lambda x)$ satisfies $u_\lambda\in H_\mu^1(\R^2)$ and 
	\begin{equation*}
		E_{\theta}(u_\lambda,\R^2)=\frac{\lambda^2}{2}\|\nabla u\|_{L^2(\R^2)}^2 -\frac{\lambda^{p-2}}{p}\|u\|_{L^p(\R^2)}^p-\frac{\lambda^{q-1}}{q}\|\tau_\theta u\|_{L^q(\R)}^q\,.
	\end{equation*}
 	When $q>3$, fixing any $u\in H_\mu^1(\R^2)$ and taking $\lambda\to+\infty$ gives again $\E_{\R^2,\theta}(\mu)=-\infty$. When $q=3$ the same can be done for sufficiently large $\mu$ taking $u\in H_\mu^1(\R^2)$ to be almost optimal in \eqref{eq:GNR2delta}, since in this case 
 	\[
 	\begin{split}
 	E_\theta(u_\lambda,\R^2)=&\,\lambda^2\left(\frac12\|\nabla u\|_{L^2(\R^2)}^2-\frac1q\|\tau_\theta u\|_{L^q(\R)}^q\right)-\frac{\lambda^{p-2}}{p}\|u\|_{L^p(\R^2)}^p\\[.2cm]
 	<&\,\frac{\lambda^2}2\|\nabla u\|_{L^2(\R^2)}^2\left(1-\frac{C_3-\varepsilon}3\sqrt{\mu}\right)\to-\infty
 	\end{split}
 	\]
 	as $\lambda\to+\infty$, provided $\mu$ is large enough.
\end{rem}

We conclude this section with the following lemmas, where we denote by $\vec{v}_\theta^\perp:=(-\sin\theta,\cos\theta)$ and by
\[
H_\theta^+:=\left\{P\in\R^2\,:\,P\cdot\vec{v}_\theta^\perp>0\right\},\qquad H_\theta^-:=\left\{P\in\R^2\,:\,P\cdot\vec{v}_\theta^\perp<0\right\}
\]
the half-planes identified by $s_\theta$ and $S_{\theta,R}$, respectively. First we point out of in which sense ground states of $E_\theta(\cdot,\R^2)$ solve \eqref{NLS_Standard} and which further features they display.

\begin{lem}
	\label{lem:reg_s}
	Let $\theta\in\left(-\frac\pi2,\frac\pi2\right]$, $u\in H_\mu^1(\R^2)$ be a ground state of $E_\theta(\cdot,\R^2)$ and denote $u^\pm:=u_{|H_\theta^\pm}$. Then, $u^\pm\in H^2(H_\theta^\pm)$ and satisfy
	\begin{equation}
		\label{side-equation}
	-\Delta u^\pm+\lambda u^\pm=|u^\pm|^{p-2}u\qquad\text{in}\quad L^2(H_\theta^\pm),
	\end{equation}
	for some suitable $\lambda\in\R$, and
	\begin{gather}
		\label{boundary-condition1}
		\tau_\theta u^+=\tau_\theta u^-=\tau_\theta u\qquad\text{in}\quad H^{\frac{3}{2}}(s_\theta),\\[.2cm]
		\label{boundary-condition2}
	\frac{\partial u^+}{\partial\vec{v}_\theta^\perp}-\frac{\partial u^-}{\partial\vec{v}_\theta^\perp}=-|\tau_\theta u|^{q-2}\tau_\theta u\qquad\text{in}\quad H^{\frac{1}{2}}(s_\theta).
	\end{gather}
Moreover, $u\in C(\R^2)\cap L^\infty(\R^2)$, is positive (up to a change of sign) and
\begin{equation}
	\label{eq-mixedsecond}
	\partial_{\vec{v}_\theta\vec{v}_\theta}^2u\in L^2(\R^2),\qquad\partial_{\vec{v}_\theta^\perp\vec{v}_\theta}^2u\in L^2(\R^2),\qquad\partial_{\vec{v}_\theta\vec{v}_\theta^\perp}^2u\in L^2(\R^2)
\end{equation}
\end{lem}

\begin{proof}
	Since $u$ is a critical point of $E_\theta(\cdot,\R^2)$ in $H_\mu^1(\R^2)$, computing the Euler-Lagrange equation of $u$, e.g., with variations in the form $\sqrt{\mu/\|u+t\varphi\|_2^2}(u+t\varphi)$, with $\varphi\in C_0^\infty(\R^2\setminus s_\theta)$, yields \eqref{side-equation} in distributional sense. By standard regularity theory, this immediately implies that $u^\pm\in H^2(H_\theta^\pm)$, and thus that \eqref{side-equation} is satisfied in $L^2(H_\theta^\pm)$, that $\tau_\theta u^\pm\in H^{\frac{3}{2}}(s_\theta)\subset C(s_\theta)$ and that $\frac{\partial u^\pm}{\partial\vec{v}_\theta^\perp}\in H^{\frac{1}{2}}(s_\theta)$. In addition, let $(u_n)_n\subset C_0^\infty(\R^2)$ be such that $u_n\to u$ in $H^1(\R^2)$ as $n\to+\infty$. Then, since the trace operator is bounded from $H^1(H_\theta^\pm)$ to $L^2(s_\theta)$,
	\[
	\begin{split}
		\|\tau_\theta u^+-\tau_\theta u^-\|_{L^2(\R)}&\,\leq\|\tau_\theta u^+ - \tau_\theta u_n\|_{L^2(\R)}+\|\tau_\theta u_n-\tau_\theta u^-\|_{L^2(\R)}\\
		&\,\lesssim\|u^+-u_n\|_{H^1(H_\theta^+)}+\|u_n-u^-\|_{H^1(H_\theta^-)}\to0\qquad\text{as }n\to+\infty
	\end{split}
	\]
	so that $\tau_\theta u^+=\tau_\theta u^-$ almost everywhere on $s_\theta$. Hence, by the regularity of $\tau_\theta u^\pm$, one obtains \eqref{boundary-condition1}. Moreover, since $u^\pm\in H^2(H_\theta^\pm)$ implies $u^\pm\in C\big(\overline{H_\theta^\pm}\big)\cap L^\infty\big(\overline{H_\theta^\pm}\big)$, there results that $u\in C(\R^2)\cap L^\infty(\R^2)$. On the other hand, \eqref{boundary-condition2} arises just computing the Euler-Lagrange equation of $u$ using, e.g., variations as before but with $s_\theta\subset\text{supp}\{\varphi\}$.
	
	At this point, it is not difficult to see that, if $u$ is a ground state, then $|u|$ is a ground state too. Thus, applying the Maximum Principle to \eqref{side-equation}, one obtains that $u$ cannot vanish on $H_\theta^\pm$. Moreover, it displays the same sign on $H_\theta^\pm$, since otherwise it should vanish on $s_\theta$, which contradicts \eqref{eq:stima}. In order to exclude the vanishing also on a proper subset of $s_\theta$, one can rely on the properties of the Steiner symmetrization of $u$ with respect to the line $s_\theta$, as explained in the following. Up to rotations, it is sufficient to address the case $\theta=0$, i.e. the case where the line $s_\theta$ coincides with the horizontal axis $\{y=0\}$. Following e.g. the reference \cite{Brock99}, one can introduce the Steiner symmetrization $u^\star$ of $u$ with respect to $\{y=0\}$, which is symmetric with respect to the line $\{y=0\}$, non-increasing on every orthogonal half-line of the form $\{x=k, y\geq 0\}$ and on such half-lines attains its maximum at the point $(k,0)$: for this reason, the term $-\frac1q\|\tau_\theta \cdot\|_{L^q(s_\theta)}^q$ does not increase passing from $u$ to $u^\star$. At the same time, the term $-\frac1p\|\cdot\|_{L^p(\R^2)}^p$ is preserved since $u$ and $u^\star$ are equimeasurable, while the term $\frac12\|\nabla\cdot\|_{L^2(\R^2)}^2$ does not increase by P\'olya-Szeg\H{o} inequality (see for istance \cite[Theorem 1]{Brock99}). Given all this, if $u$ is a ground state, then also $u^\star$ is a ground state. Moreover, if $u^\star$ vanishes at some point $(k,0)$ on $\{y=0\}$, then it vanishes also on the whole line $\{x=k\}$, but this contradicts the fact that every ground state does not vanish outside $s_\theta$, thus $u$ is positive on the whole $\R^2$, up to a change of sign. 
	
	It is, then, left to prove \eqref{eq-mixedsecond}. Up to rotations, it is sufficient to address the case $\theta=0$ (where $s_0=\R$, $H_0^\pm=\R_\pm^2$), for which the claim reduces to $\partial_{xx}^2u$, $\partial_{yx}^2u$, $\partial_{xy}^2u\in L^2(\R^2)$. We show the proof for $\partial_{yx}^2u$, the other being analogous (or simpler). Since $u\in H^1(\R^2)$, we already know that $\partial_x u\in L^2(\R^2)$. Given $\varphi\in C_c^\infty(\R^2)$ and setting $I_\varepsilon=\R\times(-\varepsilon,\varepsilon)$ for $\varepsilon>0$, we have
	\[
	\begin{split}
	\int_{\R^2}\partial_xu\partial_y\varphi\,dxdy&\,=\lim_{\varepsilon\to0}\int_{\R^2\setminus I_\varepsilon}\partial_xu\partial_y\varphi\,dxdy\\
	&\,=\lim_{\varepsilon\to0}\int_{\R_+^2\setminus I_\varepsilon}\partial_xu^+\partial_y\varphi\,dxdy+\lim_{\varepsilon\to0}\int_{\R_-^2\setminus I_\varepsilon}\partial_xu^-\partial_y\varphi\,dxdy\\
	&\,=-\lim_{\varepsilon\to0}\int_{\R_+^2\setminus I_\varepsilon}\partial_{yx}^2u^+\varphi\,dxdy-\lim_{\varepsilon\to0}\int_{\R_-^2\setminus I_\varepsilon}\partial_{yx}^2u^-\varphi\,dxdy\\
	&\,\,\quad+\lim_{\varepsilon\to0}\int_\R\left(\partial_xu^-(x,-\varepsilon)\varphi(x,-\varepsilon)-\partial_xu^+(x,\varepsilon)\varphi(x,\varepsilon)\right)\,dx\\
	&\,=-\int_{\R_+^2}\partial_{yx}^2u^+\varphi\,dxdy-\int_{\R_-^2}\partial_{yx}^2u^-\varphi\,dxdy\\
	&\,\,\quad-\lim_{\varepsilon\to0}\int_\R\left(u^-(x,-\varepsilon)\partial_x\varphi(x,-\varepsilon)-u^+(x,\varepsilon)\partial_x\varphi(x,\varepsilon)\right)\,dx\\
	&\,=-\int_{\R_+^2}\partial_{yx}^2u^+\varphi\,dxdy-\int_{\R_-^2}\partial_{yx}^2u^-\varphi\,dxdy\,,
	\end{split}
	\]
	since
	\[
	\lim_{\varepsilon\to0}\int_\R\left(u^-(x,-\varepsilon)\partial_x\varphi(x,-\varepsilon)-u^+(x,\varepsilon)\partial_x\varphi(x,\varepsilon)\right)\,dx\to0\qquad\text{as}\quad\varepsilon\to0
	\]
	by Dominated Convergence, since $\varphi$ is compactly supported and $u,\partial_x\varphi$ are continuous. Hence, 
	\[
	\partial_{yx}^2u=\begin{cases}
	\partial_{yx}^2u^+ & \text{on }\R_+^2\\[.2cm]
	\partial_{yx}^2u^- & \text{on }\R_-^2\,,
	\end{cases}
	\]
	and thus $\partial_{yx}^2u\in L^2(\R^2)$.
\end{proof}

\begin{rem}
	Note that, further developing the arguments of the proof of Lemma \ref{lem:reg_s}, one can see that \eqref{boundary-condition2} is in fact an equality in $H^{\frac{3}{2}}(s_\theta)$ and that $u^\pm\in C^1(H_\theta^\pm)$. 
\end{rem}

Finally, we mention the main features of the ground states of $E_{\theta,R}(\cdot,\R^2)$ as solutions of \eqref{NLS_Strip}. Here se omit the proof as it is analogous to proving the feature of the soliton in the standard case.

\begin{lem}
	\label{lem:reg_S}
	Let $\theta\in\left(-\frac\pi2,\frac\pi2\right]$ and $R>0$. If $u\in H_\mu^1(\R^2)$ is a ground state of $E_{\theta,R}(\cdot,\R^2)$, then $u\in H^2(\R^2)$ and satisfies \eqref{NLS_Strip} in $L^2(\R^2)$ and is positive (up to a change of sign).
\end{lem}

\begin{rem}
 Note that, arguing as in the proof of Lemma \ref{lem:reg_s}, one can check that in the case of the strip the normal derivatives at the boundary coincide, thus preventing the singular behavior exhibited by the ground states of $E_{\theta}(\cdot,\R^2)$.
\end{rem}

\section{Singular limit with $\Z^2$-periodic point nonlinearities: proof of Theorem \ref{thm:limZ2}}
\label{sec:limZ2}
This section is devoted to the proof of Theorem \ref{thm:limZ2}, namely to show that the (properly scaled) ground state problem $\E_{\G_\varepsilon}$ with $\Z^2$-periodic point nonlinearities converges in a proper sense to the two-dimensional one $\E_{\R^2}$ as $\varepsilon\to0$. 

Throughout, we set $V_\varepsilon:=\varepsilon V$, $\varepsilon>0$, to be the $\Z^2$-periodic subset of $\V_{\G_\varepsilon}$ where the point nonlinearities of $E(\cdot,\G_\varepsilon)$ are located, with $V\subset\V_{\G_1}$  the corresponding $\Z^2$-periodic set in $\G_1$. Furthermore, $\alpha, \beta $ will be fixed as in Theorem \ref{thm:limZ2}, i.e.
\[
\alpha=\frac12,\qquad\beta=\frac{\#\left(\V_{\G_1}\cap Q_0\right)}{\# V_0}\varepsilon\,,
\]
where $Q_0,V_0$ are the subsets of $\G_1$ associated to $V$ as in Remark \ref{rem:Per_dec}.

We begin with a first upper bound on the ground state level on grids with shrinking edges.
\begin{lem}
	\label{lem:uplev}
	For every $p\in(2,4)$, $q\in(2,4)$ and $\mu>0$, there results that
	\[
	\varepsilon\E_{\G_\varepsilon}\left(\frac{2\mu}{\varepsilon}\right)\leq\E_{\R^2}(\mu)+o(1)\qquad\text{as }\varepsilon\to0.
	\]
\end{lem}

\begin{proof}
	As recalled at the beginning of Section \ref{sec:R2}, when $p,q\in(2,4)$ there exists $\phi_\mu\in H_\mu^1(\R^2)$ such that $E(\phi_\mu,\R^2)=\E_{\R^2}(\mu)$ for every $\mu>0$. By standard regularity theory, $\phi_\mu\in C^\infty(\R^2)\cap H^2(\R^2)$. In addition, we can set $u_\varepsilon:={\phi_\mu}_{|\G_\varepsilon}$ and, by \cite[Lemma 4.1]{D24}, we obtain
	\begin{equation}
		\label{eq:ueps1}
		\left|\varepsilon\|u_\varepsilon'\|_{2,\varepsilon}^2-\|\nabla \phi_\mu\|_{L^2(\R^2)}^2\right|\leq C\varepsilon,\quad \left|\frac\varepsilon2\|u_\varepsilon\|_{r,\varepsilon}^r-\|\phi_\mu\|_{L^r(\R^2)}^r\right|\leq C\varepsilon\quad\forall r\geq2, \qquad\text{as}\quad\varepsilon\to0,
	\end{equation}
	which in particular entails $u_\varepsilon\in H^1(\G_\varepsilon)$.	Moreover, since 
	\[
	\begin{split}
	&\,\left|\|\phi_\mu\|_{L^q(\R^2)}^q-\varepsilon^2\frac{\#\left(\V_{\G_1}\cap Q_0\right)}{\# V_0}\sum_{\vv\in V_\varepsilon}|u_\varepsilon(\vv)|^q\right|\\
	&\qquad\qquad\leq \left|\|\phi_\mu\|_{L^q(\R^2)}^q-\frac\varepsilon2\|u_\varepsilon\|_{q,\varepsilon}^q\right|+\varepsilon\left|\frac12\|u_\varepsilon\|_{q,\varepsilon}^q-\varepsilon\frac{\#\left(\V_{\G_1}\cap Q_0\right)}{\# V_0}\sum_{\vv\in V_\varepsilon}|u_\varepsilon(\vv)|^q\right|\,,
	\end{split}
	\]
	by \eqref{eq:ueps1} and Lemma \ref{l.GN-grid-i} we have
	\begin{multline*}
		\,\left|\|\phi_\mu\|_{L^q(\R^2)}^q-\varepsilon^2\frac{\#\left(\V_{\G_1}\cap Q_0\right)}{\# V_0}\sum_{\vv\in V_\varepsilon}|u_\varepsilon(\vv)|^q\right|\,\leq o(1)+\varepsilon^2\|u_\varepsilon\|_{2(q-1),\varepsilon}^{q-1}\|u_\varepsilon'\|_{2,\varepsilon}\\[.2cm]
		=o(1)+\varepsilon\|\phi_\mu\|_{L^2(q-1)(\R^2)}^{q-1}\|\nabla\phi_\mu\|_{L^2(\R^2)}+o(\varepsilon)=o(1)\qquad\text{as}\quad\varepsilon\to0.
	\end{multline*}
	Now, define the function $v_\varepsilon\in H_{\frac{2\mu}{\varepsilon}}^1(\G_\varepsilon)$ as
	\[
	v_\varepsilon= \sqrt{\frac{2\mu}{\varepsilon\|u_\varepsilon\|_2^2}}\,u_\varepsilon.
	\]
	Coupling the previous estimate with \eqref{eq:ueps1} entails
	\begin{multline*}
	\varepsilon\E_{\G_\varepsilon}\left(\frac{2\mu}\varepsilon\right)\,\leq \varepsilon E(v_\varepsilon,\G_\varepsilon)\\
	= \frac\varepsilon2\,\frac{2\mu}{\varepsilon\|u_\varepsilon\|_{2,\varepsilon}^2}\|u_\varepsilon'\|_{2,\varepsilon}^2-\frac\varepsilon{2p}\left(\frac{2\mu}{\varepsilon\|u_\varepsilon\|_{2,\varepsilon}^2}\right)^\frac p2\|u_\varepsilon\|_{p,\varepsilon}^p-\varepsilon^2\frac{\#\left(\V_{\G_1}\cap Q_0\right)}{q\#V_0}\left(\frac{2\mu}{\varepsilon\|u_\varepsilon\|_{2,\varepsilon}^2}\right)^\frac q2\sum_{\vv\in V_\varepsilon}|u_\varepsilon(\vv)|^q\\
	\,\,\,=(1+o(1))\left(\frac\varepsilon2\|u_\varepsilon'\|_{2,\varepsilon}^2-\frac{\varepsilon}{2p}\|u_\varepsilon\|_{p,\varepsilon}^p-\varepsilon^2\frac{\#\left(\V_{\G_1}\cap Q_0\right)}{q\#V_0}\sum_{\vv\in V_\varepsilon}|u_\varepsilon(\vv)|^q\right)=E(\phi_\mu,\R^2)+o(1).
	\end{multline*}
	
\end{proof}

The upper bound we established above enables one to establish the following a priori estimates on ground states on $\G_\varepsilon$.

\begin{lem}
	\label{lem:aprZ2}
	For every $p\in(2,4)$, $q\in(2,4)$ and $\mu>0$ there exists $M>0$ such that
	\[
	\frac{1}{M}\le \varepsilon\|u^\prime_\varepsilon\|_{2,\varepsilon}^2\,,\,\varepsilon\|u_\varepsilon\|_{p,\varepsilon}^p\,,\,\varepsilon^2\sum_{\vv\in V_\varepsilon}|u_\varepsilon(\vv)|^q\le M\,,
	\]
	for every $\varepsilon>0$ and every ground state $u_\varepsilon$ of $ \E_{\G_\varepsilon}\left(\frac{2\mu}{\varepsilon}\right)$.
\end{lem}
\begin{proof}
	By Theorem \ref{thm:GSZ2} we have $\E_{\G_\varepsilon}\left(\frac{2\mu}{\varepsilon}\right)<0$ for every $p,q\in(2,4)$ and every $\mu>0$ and $\varepsilon>0$. If $u_\varepsilon\in H_{\frac{2\mu}{\varepsilon}}^1(\G_\varepsilon)$ satisfies $E(u_\varepsilon,\G_\varepsilon)=\E_{\G_\varepsilon}\left(\frac{2\mu}{\varepsilon}\right)$, the negativity of the ground state level yields
	\begin{equation}
	\label{eq:negE}
	\frac{\varepsilon}{2}\|u^\prime_\varepsilon\|_{2,\varepsilon}^2<\frac{\varepsilon}{2p}\|u_\varepsilon\|_{p,\varepsilon}^p+\varepsilon^2\frac{\#\left(\V_{\G_1}\cap Q_0\right)}{q\#V_0}\sum_{\vv\in V_\varepsilon}|u_\varepsilon(\vv)|^q\,. 
	\end{equation}
	Now, by \eqref{GN-eps.2}
	\begin{equation}
		\label{proof.eps-bound.p}
		\frac{\varepsilon}{2p}\|u_\varepsilon\|_{p,\varepsilon}^p\le \varepsilon C_p\varepsilon^{\frac{p}{2}-1}\|u_\varepsilon\|_{2,\varepsilon}^2\|u_\varepsilon^\prime\|_{2,\varepsilon}^{p-2}=2C_p\mu\left(\varepsilon\|u_\varepsilon^\prime\|_{2,\varepsilon}^2\right)^{\frac{p}{2}-1},
	\end{equation}
	for a suitable constant $C_p>0$ depending only on $p$, whereas by \eqref{eq-GNgrid2per}
	\begin{equation}
	\label{eq:upes2}
	\varepsilon^2\sum_{\vv\in V_\varepsilon}|u_\varepsilon(\vv)|^q\lesssim \varepsilon\|u_\varepsilon\|_{q,\varepsilon}^q +\varepsilon^2\|u_\varepsilon\|_{2(q-1),\varepsilon}^{q-1}\|u_\varepsilon^\prime\|_{2,\varepsilon}. 
	\end{equation}
	Since by \eqref{GN-eps.1} we have
	\[
	\varepsilon^2\|u_\varepsilon\|_{2(q-1),\varepsilon}^{q-1}\|u_\varepsilon^\prime\|_{2,\varepsilon}\lesssim \varepsilon^2\|u_\varepsilon\|_{2,\varepsilon}^{\frac{q}{2}}\|u_\varepsilon^\prime\|_{2,\varepsilon}^{\frac{q}{2}}=\mu^{\frac{q}{4}}\varepsilon^\frac{4-q}{2}\left(\varepsilon\|u_\varepsilon^\prime\|_{2,\varepsilon}^2\right)^{\frac{q}{4}}\,,
	\]
	and, arguing as in \eqref{proof.eps-bound.p}, it holds
	\[
	 \varepsilon
	\|u_\varepsilon\|_{q,\varepsilon}^q\lesssim \left(\varepsilon\|u_\varepsilon^\prime\|_{2,\varepsilon}^2\right)^{\frac{q}{2}-1}\,,
	\]
	combining with \eqref{eq:upes2} leads to
	\begin{equation}
		\label{proof.eps-bound.q}
		\varepsilon^2\sum_{\vv\in V_\varepsilon}|u_\varepsilon(\vv)|^q\lesssim (\varepsilon\|u_\varepsilon^\prime\|_{2,\varepsilon}^2)^\frac{q-2}{2} + \varepsilon^\frac{4-q}{2}(\varepsilon\|u^\prime_\varepsilon\|_{2,\varepsilon}^2)^{\frac{q}{4}}\,.
	\end{equation} 
	Hence, letting $t:=\varepsilon\|u_\varepsilon^\prime\|_{2,\varepsilon}^2$ and combining \eqref{eq:negE}, \eqref{proof.eps-bound.p} and \eqref{proof.eps-bound.q} gives
	\[
	\frac{1}{2}t\le 2C_p\mu t^\frac{p-2}{2} + c_1t^\frac{q-2}{2} + c_2\varepsilon^\frac{4-q}{2}t^\frac{q}{4}
	\]
	for suitable constants $c_1,c_2>0$, that together with $p\in(2,4)$, $q\in(2,4)$ implies
	\[
	\varepsilon\|u^\prime\|_{2,\varepsilon}^2\lesssim 1\,,
	\]
	in turn yielding 
	\[
	\varepsilon\|u_\varepsilon\|_{p,\varepsilon}^p\,,\,\varepsilon^2\sum_{\vv\in V_\varepsilon}|u_\varepsilon(\vv)|^q\lesssim 1
	\]
	by \eqref{proof.eps-bound.p} and \eqref{proof.eps-bound.q}.
	
	As for the estimates from below, note first that, by Lemma \ref{lem:uplev}, there exists $K>0$ independent of $\varepsilon$ such that $\varepsilon E(u_\varepsilon,\G_\varepsilon)<-K$. Denoting again $t:=\varepsilon\|u_\varepsilon^\prime\|_{2,\varepsilon}^2$ and exploiting \eqref{proof.eps-bound.p} and \eqref{proof.eps-bound.q}, we then have
	\begin{equation}\label{proof.bound.}
		K<\frac{\varepsilon}{2p}\|u_\varepsilon\|_{p,\varepsilon}^p+\varepsilon^2\frac{\#\left(\V_{\G_1}\cap Q_0\right)}{q\#V_0}\sum_{\vv\in V_\varepsilon}|u_\varepsilon(\vv)|^q\lesssim t^\frac{p-2}{2} + t^\frac{q-2}{2} + t^\frac{q}{4},
	\end{equation}
	that is $ \varepsilon\|u_\varepsilon^\prime\|_{2,\varepsilon}^2\gtrsim1. $
	Assume now by contradiction that either
	\[
	\liminf_{\varepsilon\to0^+}\varepsilon\|u_\varepsilon\|_{p,\varepsilon}^p=0, \qquad\text{or}\qquad\liminf_{\varepsilon\to0^+}\varepsilon^2\sum_{\vv\in V_\varepsilon}|u_\varepsilon(\vv)|^q=0\,.
	\]
	Then (recalling again \eqref{eq-GNgrid2per}) we would have
	\[
	\begin{split}
		&\liminf_{\varepsilon\to0}\varepsilon\E_{\G_\varepsilon}\left(\frac{2\mu}\varepsilon\right)\\[.2cm]
		&\quad\geq \min\left\{\liminf_{\varepsilon\to0}\left(\frac\varepsilon2\|u_\varepsilon'\|_{2,\varepsilon}^2-\frac{\varepsilon}{2p}\|u_\varepsilon\|_{p,\varepsilon}^p\right), \liminf_{\varepsilon\to0}\left(\frac\varepsilon2\|u_\varepsilon'\|_{2,\varepsilon}^2-\varepsilon^2\frac{\#\left(\V_{\G_1}\cap Q_0\right)}{q\#V_0}\sum_{\vv\in V_\varepsilon}|u_\varepsilon(\vv)|^q\right)\right\}\\[.2cm]
		&\quad\geq \min\left\{\liminf_{\varepsilon\to0}\varepsilon\left(\frac12\|u_\varepsilon'\|_{2,\varepsilon}^2-\frac1{2p}\|u_\varepsilon\|_{p,\varepsilon}^p\right), \liminf_{\varepsilon\to0}\varepsilon\left(\frac12\|u_\varepsilon'\|_{2,\varepsilon}^2-\frac1{2q}\|u_\varepsilon\|_{q,\varepsilon}^q\right)\right\}\\[.2cm]
		&\quad\geq \min\left\{\liminf_{\varepsilon\to0}\varepsilon\inf_{v\in H_{\frac{2\mu}{\varepsilon}}^1(\G_\varepsilon)}\left(\frac12\|v'\|_{2,\varepsilon}^2-\frac1{2p}\|v\|_{p,\varepsilon}^p\right), \liminf_{\varepsilon\to0}\varepsilon\inf_{v\in H_{\frac{2\mu}{\varepsilon}}^1(\G_\varepsilon)}\left(\frac12\|v'\|_{2,\varepsilon}^2-\frac1{2q}\|v\|_{q,\varepsilon}^q\right)\right\}\,.
	\end{split}
	\]
	However, since by \cite[Theorem 2.2]{D24} we known that
	\[
	\lim_{\varepsilon\to0}\varepsilon\inf_{v\in H_{\frac{2\mu}{\varepsilon}}^1(\G_\varepsilon)}\left(\frac12\|v'\|_{2,\varepsilon}^2-\frac1{2r}\|v\|_{r,\varepsilon}^r\right)=\inf_{w\in H_\mu^1(\R^2)}\overline{E}_r(w,\R^2)
	\]
	for every $r\in(2,4)$, with $\overline E_r(\cdot,\R^2)$ defined by \eqref{eq:Ebar}, combining with Lemma \ref{lem:uplev} and \eqref{eq:2<1} would imply
	\[
	\E_{\R^2}(\mu)\geq\liminf_{\varepsilon\to0}\varepsilon\E_{\G_\varepsilon}\left(\frac{2\mu}{\varepsilon}\right)\geq\min\left\{\inf_{w\in H_\mu^1(\R^2)}\overline{E}_p(w,\R^2),\,\inf_{w\in H_\mu^1(\R^2)}\overline{E}_q(w,\R^2)\right\}>\E_{\R^2}(\mu)\,,
	\]
	i.e. the contradiction that allow us to conclude.
\end{proof}

\begin{proof}[Proof of Theorem \ref{thm:limZ2}]
	Let $u_\varepsilon\in H_{\frac{2\mu}{\varepsilon}}^1(\G_\varepsilon)$ be such that $E(u_\varepsilon,\G_\varepsilon)=\E_{\G_\varepsilon}\big(\frac{2\mu}{\varepsilon}\big)$, and let $\mathcal{A}u_\varepsilon:\R^2\to\R$ be its piecewise affine extension to $\R^2$ as defined in \cite[Section 2]{D24}. Combining Lemma \ref{lem:aprZ2} with \cite[Lemma 6.1]{D24} gives immediately
	\begin{equation}
		\label{eq:AL2Lp}
	\begin{split}
		\left|\|\mathcal{A}u_\varepsilon\|_{L^2(\R^2)}^2-\frac\varepsilon2\|u_\varepsilon\|_{2,\varepsilon}^2\right|\lesssim\varepsilon, \quad 	\left|\|\mathcal{A}u_\varepsilon\|_{L^p(\R^2)}^p-\frac\varepsilon2\|u_\varepsilon\|_{p,\varepsilon}^p\right|\lesssim\varepsilon\quad\text{as }\varepsilon\to0\,,
	\end{split}
	\end{equation}
	whereas the definition of $\mathcal{A}u_\varepsilon$ ensures (see e.g. \cite[Lemma 4.4]{D24})
	\begin{equation}
	\label{eq:gradA}
	\|\nabla \mathcal{A}u_\varepsilon\|_{L^2(\R^2)}^2\lesssim\varepsilon\|u_\varepsilon'\|_{2,\varepsilon}^2\,.
	\end{equation}
	Moreover, since $u_\varepsilon$ and $\mathcal{A}u_\varepsilon$ coincide on $\V_{\G_\varepsilon}$, denoting by $\widetilde{u}_\varepsilon$ the restriction of $\mathcal{A}u_\varepsilon$ to $\G_\varepsilon$ and using Lemma \ref{l.GN-grid-i} and Lemma \ref{l.GN-eps.} we obtain
	\[
	\begin{split}
	\left|\frac\varepsilon2\|\widetilde{u}_\varepsilon\|_{q,\varepsilon}^q-\varepsilon^2\frac{\#\left(\V_{\G_1}\cap Q_0\right)}{\#V_0}\sum_{\vv\in V_\varepsilon}|u_\varepsilon(\vv)|^q\right|&\,=\left|\frac\varepsilon2\|\widetilde{u}_\varepsilon\|_{q,\varepsilon}^q-\varepsilon^2\frac{\#\left(\V_{\G_1}\cap Q_0\right)}{\#V_0}\sum_{\vv\in V_\varepsilon}|\widetilde{u}_\varepsilon(\vv)|^q\right|\\[.2cm]
	&\,\lesssim \varepsilon^2\|\widetilde{u}_\varepsilon\|_{2(q-1),\varepsilon}^{q-1}\|\widetilde{u}_\varepsilon^\prime\|_{2,\varepsilon}\lesssim\varepsilon^{\frac q2+1}\|\widetilde{u}_\varepsilon\|_{2,\varepsilon}\|\widetilde{u}_\varepsilon'\|_{2,\varepsilon}^{q-1}\\[.2cm]
	&\,\lesssim\varepsilon^{\frac q2+1}\|u_\varepsilon\|_{2,\varepsilon}\|u_\varepsilon'\|_{2,\varepsilon}^{q-1}\lesssim \varepsilon\quad\text{as }\varepsilon\to0
	\end{split}
	\]
	(the inequalities $\|\widetilde{u}_\varepsilon\|_{2,\varepsilon}\leq(1+o(1))\|u_\varepsilon\|_{2,\varepsilon}$, $\|\widetilde{u}_\varepsilon'\|_{2,\varepsilon}\leq\|u_\varepsilon'\|_{2,\varepsilon}$ following directly by the definition of $\widetilde{u}_\varepsilon$, see e.g. \cite[Eq. (26) and Lemmas 4.2-4.3]{D24}). Similarly, arguing as in the first part of the proof of \cite[Lemma 4.1]{D24}, since $\mathcal{A}u_\varepsilon\in H^1(\R^2)$ we have
	\begin{multline*}
	\left|\|\mathcal{A}u_\varepsilon\|_{L^q(\R^2)}^q-\frac\varepsilon2\|\widetilde{u}_\varepsilon\|_{q,\varepsilon}^q\right|\lesssim\varepsilon\|\mathcal{A}u_\varepsilon\|_{L^{2(q-1)}(\R^2)}^{q-1}\|\nabla\mathcal{A}u_\varepsilon\|_{L^2(\R^2)}\\[.2cm]
	\lesssim\varepsilon\|\mathcal{A}u_\varepsilon\|_{L^2(\R^2)}\|\nabla\mathcal{A}u_\varepsilon\|_{L^2(\R^2)}^{q-1}\lesssim\varepsilon^{\frac q2+1}\|u_\varepsilon\|_{2,\varepsilon}\|u_\varepsilon'\|_{2,\varepsilon}^{q-1}\lesssim\varepsilon\,,
	\end{multline*}
	where we made use of \eqref{eq:AL2Lp}, \eqref{eq:gradA} and Lemma \ref{lem:aprZ2}. Summing up, we have
	\[
	\left|\|\mathcal{A}u_\varepsilon\|_q^q-\varepsilon^2\frac{\#\left(\V_{\G_1}\cap Q_0\right)}{\#V_0}\sum_{\vv\in V_\varepsilon}|u_\varepsilon(\vv)|^q\right|\lesssim\varepsilon\qquad\text{as }\varepsilon\to0\,,
	\]
	that, setting 
	\[
	v_\varepsilon:=\sqrt{\frac\mu{\|\mathcal{A}u_\varepsilon\|_2^2}}\,\mathcal{A}u_\varepsilon
	\]
	and coupling with \eqref{eq:AL2Lp}, \eqref{eq:gradA} and Lemma \ref{lem:uplev} gives
	\[
	\E_{\R^2}(\mu)\leq\liminf_{\varepsilon\to0}E(v_\varepsilon,\R^2)\leq\liminf_{\varepsilon\to0}\varepsilon E(u_\varepsilon,\G_\varepsilon)=\liminf_{\varepsilon\to0}\varepsilon\E_{\G_\varepsilon}\left(\frac{2\mu}{\varepsilon}\right)\le\E_{\R^2}(\mu)\,.
	\]
	Hence, $\left(v_\varepsilon\right)_\varepsilon\subset H_\mu^1(\R^2)$ is a minimizing sequence for $\E_{\R^2}(\mu)$. Since $E(\cdot,\R^2)$ is invariant by translations, with no loss of generality we can assume that for every $\varepsilon$ it holds
	\begin{equation}
		\label{eq:AL2R0}
	\|v_\varepsilon\|_{L^2(R_{0,0})}=\max_{(i,j)\in\Z^2}\|v_\varepsilon\|_{L^2(R_{i,j})}\,,
	\end{equation}
	where $R_{i,j}=\left(-\frac12+i,\frac12+i\right)\times\left(-\frac12+j,\frac12+j\right)$ for every $(i,j)\in\Z^2$. It is then standard to show that, up to subsequences, $v_\varepsilon\to u$ in $H^1(\R^2)$ as $\varepsilon\to0$, with $E(u,\R^2)=\E_{\R^2}(\mu)$, and, by definition of $v_\varepsilon$, also $\mathcal{A}u_\varepsilon\to u$ in $H^1(\R^2)$. To this end, it is enough as usual to show that the weak limit $u$ of $v_\varepsilon$ satisfies $\|u\|_2^2=\mu$. To see that $\|u\|_2^2\not\in(0,\mu)$ one exploits the same argument based on the Brezis-Lieb lemma already used e.g. in the proof of Lemma \ref{l.ex.crit.} above. Furthermore, to rule out the case $u\equiv0$ on $\R^2$, it is sufficient to note that, if this were the case, by standard two-dimensional Gagliardo-Nirenberg inequalities on bounded sets of $\R^2$ and \eqref{eq:AL2R0} we would have
	\[
	\begin{split}
	\|v_\varepsilon\|_4^4=&\,\sum_{(i,j)\in\Z^2}\|v_\varepsilon\|_{L^4(R_{i,j})}^4\lesssim\sum_{(i,j)\in\Z^2}\|v_\varepsilon\|_{L^2(R_{i,j})}^2\|v_\varepsilon\|_{H^1(R_{i,j})}^2\\
	&\,\leq\|v_\varepsilon\|_{L^2(R_{0,0})}^2\|v_\varepsilon\|_{H^1(\R^2)}^2\to0\qquad\text{as }\varepsilon\to0\,,
	\end{split}
	\]
	and thus by interpolation $\|v_\varepsilon\|_r\to0$ for every $r\in(2,4)$, so that 
	\[ \E_{\R^2}(\mu)=\lim_{\varepsilon\to0}E(v_\varepsilon,\R^2)=\lim_{\varepsilon\to0}\frac12\|\nabla v_\varepsilon\|_2^2\geq0\,,
	\]
	which is a contradiction. Observe that, since each ground state of $\E_{\R^2}(\mu)$ is positive and radially symmetric non-increasing in $\R^2$ with respect to the point where it attains its $L^\infty$ norm, up to a possible further translations we obtain that the convergence of $\mathcal{A}u_\varepsilon(\cdot-x_\varepsilon)$ is to a ground state of $\E_{\R^2}$ with the required features.
\end{proof}

\section{Singular limit with $\Z$-periodic nonlinearities: proof of Theorems \ref{thm:limZ_theta}--\ref{thm:limZ_S}}
\label{sec:limZ}
Within this section we discuss the asymptotic behaviour of ground states of $E(\cdot,\G_\varepsilon)$ on grids $\G_\varepsilon$ with $\Z$-periodic point nonlinearities, proving the convergence to the limit problems $\E_{\R^2,\theta}$ and $\E_{\R^2,\theta,R}(\mu)$ as in Theorem \ref{thm:limZ_theta} and Theorem \ref{thm:limZ_S}, respectively. 

\subsection{The limit problem $\E_{\R^2,\theta}$: proof of Theorem \ref{thm:limZ_theta}}
In what follows, we take $V\subset\V_{\G_1}$ to be a fixed $\Z$-periodic set in $\G_1$, and we let $\vec{v}=(v_1,v_2)\in\Z^2$ be the vector associated to $V$ as in Definition \ref{def:Z-per}, $V_0\subset\G_1$ be the set associated to $V$ as in Remark \ref{rem:Per_dec}, and 
\[
\theta:=\begin{cases}
\text{\normalfont arctan}\frac{v_2}{v_1} & \text{if }v_1\neq0\\
\frac\pi2 & \text{if }v_1=0\,.
\end{cases}
\]
For every $\varepsilon>0$ we then set $V_\varepsilon:=\varepsilon V\subset\V_{\G_\varepsilon}$ and
\begin{equation}
\label{eq:abZ}
\alpha=\frac12\,,\qquad\beta=\frac{|\vec{v}|}{\#V_0}\,.
\end{equation}
The strategy of the proof of Theorem \ref{thm:limZ_theta} is analogous to that already developed in the previous section for Theorem \ref{thm:limZ2}. The main element of novelty comes from the fact that the limit problem $\E_{\R^2,\theta}$ now involves a singular term concentrated on a line. This not only forces us to derive new estimates to compare the point nonlinearities of ground states on grids with $L^q$ norms restricted to $s_\theta$ of their extensions in the plane, but it also requires an additional care whenever using the ground states of $\E_{\R^2,\theta}$, as they do not belong to $H^2(\R^2)$ (contrary to the ground states  of $\E_{\R^2}$).

Here, we will argue as follows. First, we will prove Theorem \ref{thm:limZ_theta} in the special case of
\begin{equation}
\label{eq:Vsimpl}
V=\Z\vec{v}=\left\{k\vec{v}\,:\,k\in\Z\right\},
\end{equation}
that is when the point nonlinearities are located only at the intersection between $\V_{\G_\varepsilon}$ and the line $s_\theta$. Second, we will show hot to reduce the problem with a general $\Z$-periodic set $V$ to the previous case, thus completing the proof of Theorem \ref{thm:limZ_theta}.

Assume, then, from now on that $V$ is as \eqref{eq:Vsimpl}, so that $V_0$ contains only the origin, and therefore the energy functional we are considering on $\G_\varepsilon$ becomes
\begin{equation}
\label{eq:EVsimpl}
E(u,\G_\varepsilon)=\frac12\|u'\|_{2,\varepsilon}^2-\frac1{2p}\|u\|_{p,\varepsilon}^p-\frac{|\vec{v}|}{q}\sum_{i\in\Z}|u(\varepsilon i\vec{v})|^q\,.
\end{equation}
We begin the discussion with the next preliminary estimate connecting the derivative of $\tau_\theta(\mathcal{A}u)$ along $s_\theta$ with that of the original function $u$ on $\G_\varepsilon$. 
\begin{lem}
	\label{l.eps.bound.grid.es.Zper.}
	For every $\varepsilon>0$ and every $u\in H^1(\G_\varepsilon)$, there results
	\[
	\big\|\big(\tau_\theta(\mathcal{A}u)\big)^\prime\big\|_{L^2(s_\theta)}^2\lesssim \varepsilon\|u^\prime\|_{2,\varepsilon}^2\,.
	\]
\end{lem}
\begin{proof}
	For every $(i,j)\in\Z^2$, let $U_{ij}^\varepsilon$, $D_{ij}^\varepsilon$ be the sets defined in \eqref{eq:UD}. Given $(i,j)\in\Z^2$, it can be easily seen by the definition of $\mathcal{A}u$ that
	$$ \mathcal{A}u(x)=\left(-\frac{x_1}{\varepsilon}+i+1\right)u(\varepsilon i,\varepsilon j) + \left(\frac{x_2}{\varepsilon}-j\right)u(\varepsilon(i+1),\varepsilon(j+1)) + \left(\frac{x_1-x_2}{\varepsilon}-i+j\right)u(\varepsilon(i+1),\varepsilon j) $$
	for every $x=(x_1,x_2)\in D_{ij}^\varepsilon$, and
	$$ \mathcal{A}u(x)= \left(-\frac{x_2}{\varepsilon}+j+1\right)u(\varepsilon i,\varepsilon j) + \left(\frac{x_1}{\varepsilon}-i\right)u(\varepsilon(i+1),\varepsilon(j+1)) + \left(\frac{x_2-x_1}{\varepsilon}+i-j\right)u(\varepsilon i,\varepsilon(j+1))$$
	for every $x\in U_{ij}^\varepsilon$.
	
	Let now $(i,j)\in\Z^2$ be such that $|s_\theta\cap D_{ij}^\varepsilon|>0$, so that there exist $a<b$, depending on $i,j$ and such that $b-a\lesssim\varepsilon$, for which $s_\theta\cap D_{ij}^\varepsilon$ can be parameterized as $s\vec{v}$, with $s\in[a,b]$. It then follows
	\[
	\begin{split}
	\int_{s_\theta\cap D_{ij}^\varepsilon}&\,\big|\big(\tau_\theta(\mathcal{A}u)\big)^\prime\big|^2\,ds\\[.2cm]
	&\,=\left|-\frac{v_1}{\varepsilon}u(\varepsilon i,\varepsilon j) + \frac{v_2}{\varepsilon}u(\varepsilon(i+1),\varepsilon(j+1)) + \frac{v_1-v_2}{\varepsilon}u(\varepsilon(i+1),\varepsilon j)\right|^2\frac{(b-a)}{|\vec{v}|^2}\\[.2cm]
	&\,\lesssim(b-a)\int_{\G_\varepsilon\cap D_{ij}^\varepsilon}|u'|^2\,dx\lesssim\varepsilon\|u'\|_{L^2(\G_\varepsilon\cap D_{ij}^\varepsilon)}^2\,,
	\end{split}
	\]
	where we estimated $\frac{1}{\varepsilon^2}\left|u(\varepsilon i,\varepsilon j)-u(\varepsilon (i+1), \varepsilon j)\right|^2$ and $\frac{1}{\varepsilon^2}\left|u(\varepsilon(i+1),\varepsilon(j+1))-u(\varepsilon (i+1), \varepsilon j)\right|^2$ with $\|u'\|_{L^2(\G_\varepsilon\cap D_{ij}^\varepsilon)}^2$. Since an analogous estimate holds whenever $s_\theta$ intersects $U_{ij}^\varepsilon$ for some $(i,j)\in\Z^2$, summing over all values of indices for which $s_\theta$ intersects either $D_{ij}^\varepsilon$ or $U_{ij}^\varepsilon$ we conclude.
\end{proof}
We can now start to develop the proof of Theorem \ref{thm:limZ_theta} for $E(\cdot,\G_\varepsilon)$ as in \eqref{eq:EVsimpl}. We begin with an upper bound on the ground state level and suitable a priori estimates on the ground states.

\begin{lem}
	\label{l.en.upperbound-Z}
	For every $p\in(2,4)$, $q\in(2,3)$ and $\mu>0$ there results
	\[
	\varepsilon\mathcal{E}_{\G_\varepsilon}\left(\frac{2\mu}{\varepsilon}\right)\le \mathcal{E}_{\R^2,\theta}(\mu) + o(1)\qquad\text{as }\varepsilon\to0\,. 
	\]
\end{lem}
\begin{proof}
	The argument is analogous to the one in the proof of Lemma \ref{lem:uplev}, replacing  a ground state $\phi_\mu$ of $\E_{\R^2}(\mu)$ with a ground state $\psi_\mu$ of $\E_{\R^2,\theta}(\mu)$, whose existence is guaranteed for every $p\in(2,4)$, $q\in(2,3)$ and $\mu>0$ by Theorem \ref{thm:exGS_theta}. Hence, letting $u_\varepsilon$ be the restriction of $\psi_\mu$ to $\G_\varepsilon$, if one establishes that $u_\varepsilon\in H^1(\G_\varepsilon)$ and that
	\begin{equation}
		\label{eq:GSthetaG}
		\left\{
		\begin{split}
		&\|\psi_\mu\|_{L^r(\R^2)}^r=\frac\varepsilon2\|u_\varepsilon\|_{r,\varepsilon}^r+o(1),\quad\forall r\geq2\,,\\[.2cm]
		&\|\tau_\theta\psi_\mu\|_{L^q(s_\theta)}^q=\varepsilon|\vec{v}|\sum_{i\in\Z}|u_\varepsilon(\varepsilon i\vec{v})|^q+o(1)\,,\\[.2cm]
		&\,\|\nabla\psi_\mu\|_{L^2(\R^2)}^2=\varepsilon\|u_\varepsilon'\|_{2,\varepsilon}^2+o(1)\\
		\end{split}
		\right.\qquad\text{as}\quad\varepsilon\to0.
	\end{equation}
	then the lemma is proved. 
	
	Note first that the restriction of $u_\varepsilon$ to any vertical and horizontal line in $\G_\varepsilon$ is a well-defined function in $H^1(\R)$. Indeed, if $\theta\not\in\left\{0,\frac\pi2\right\}$, each such line splits into the union of a part in $H_\theta^+$ and a part in $H_\theta^-$, and the restriction of $\psi_\mu$ to both half-lines is an $H^{\frac{3}{2}}$ function since $\psi_\mu\in H^2(H_\theta^\pm)$ by Lemma \ref{lem:reg_s}. Thus, such restrictions to the two half-lines are $H^1$ functions and this fact, combined with the global continuity of $\psi_\mu$ (see again Lemma \ref{lem:reg_s}), entails that the restriction to the whole line belongs to $H^1(\R)$. The same is true when $\theta=0$ or $\theta=\frac\pi2$ for all lines but the $x$-axis or the $y$-axis, that in these cases correspond to $s_\theta$, respectively. However, when $\theta=0$ Lemma \ref{lem:reg_s} ensures that $\partial_x\psi_\mu\in H^1(\R^2)$, so that its trace on the $x$-axis is in $H^{\frac{1}{2}}(\R)$, and thus the trace of $\psi_\mu$ is again in $H^{3/2}(\R)$, and the same is true for $\partial_y\psi_\mu$ when $\theta=\frac\pi2$. Morever, $u_\varepsilon$ is continuous on $\G_\varepsilon$, again by the continuity of $\psi_\mu$.
	
	Now, since $\psi_\mu\in H^1(\R^2)\cap L^\infty(\R^2)$, the first line of \eqref{eq:GSthetaG} is given by \cite[Lemma 4.1]{D24} (note that this is sufficient to get \cite[Eq. (21)]{D24}). As for the second line, we compute
	\[
	\begin{split}
		\left|\|\tau_\theta\psi_\mu\|_{L^q(s_\theta)}^q-\varepsilon|\vec{v}|\sum_{i\in\Z}|u_\varepsilon(\varepsilon i\vec{v})|^q\right|=&\, \left|\sum_{i\in\Z}\int_{\varepsilon i|\vec{v}|}^{\varepsilon(i+1)|\vec{v}|}\left|\psi_\mu\left(\frac{s\vec{v}}{|\vec{v}|}\right)\right|^q\,ds-\varepsilon|\vec{v}|\sum_{i\in\Z}|\psi_\mu(\varepsilon i\vec{v})|^q\right|\\
		\leq&\,\sum_{i\in\Z}\left|\int_{\varepsilon i|\vec{v}|}^{\varepsilon(i+1)|\vec{v}|}\left(\left|\psi_\mu\left(\frac{s\vec{v}}{|\vec{v}|}\right)\right|^q-|\psi_\mu(\varepsilon i\vec{v})|^q\right)\,ds\right|\\
		\lesssim&\,\varepsilon\sum_{i\in\Z}\int_{\varepsilon i|\vec{v}|}^{\varepsilon(i+1)|\vec{v}|}|\psi_\mu|^{q-1}|\partial_{\vec{v}_\theta}\psi_\mu|\,ds=\varepsilon\int_{s_\theta}|\psi_\mu|^{q-1}|\partial_{\vec{v}_\theta}\psi_\mu|\,ds\\
		\leq&\,\varepsilon\|\psi_\mu\|_{L^{2(q-1)}(s_\theta)}^{q-1}\|\partial_{\vec{v}_\theta}\psi_\mu\|_{L^2(s_\theta)}\lesssim\varepsilon\,,
	\end{split}
	\]
	where the last inequality relies on the fact that $\partial_{\vec{v}_\theta}\psi_\mu\in H^1(\R^2)$ (guaranteed again by \eqref{eq-mixedsecond}).
	
	It is left to prove the third line of \eqref{eq:GSthetaG}. To this end, for every $(i,j)\in\Z^2$ denote by $Q_{ij}^\varepsilon:=U_{ij}^\varepsilon\cup D_{ij}^\varepsilon=[\varepsilon i,\varepsilon(i+1)]\times[\varepsilon j,\varepsilon(j+1)]$ the corresponding square identified by $\G_\varepsilon$ in $\R^2$. Assume first that $Q_{ij}^\varepsilon\cap s_\theta=\emptyset$, that is for instance $Q_{ij}^\varepsilon\subset H_\theta^+$ (the other case is analogous). Denoting by $b_{ij}^\varepsilon:=[\varepsilon i,\varepsilon(i+1)]\times\left\{\varepsilon j\right\}$, we have 
	\[
	\begin{split}
	\left|\varepsilon\|u_\varepsilon'\|_{L^2(b_{ij}^\varepsilon)}^2-\|\partial_x\psi_\mu\|_{L^2(Q_{ij}^\varepsilon)}^2\right|&\,=\left|\varepsilon\int_{\varepsilon i}^{\varepsilon(i+1)}|u_\varepsilon'(x,\varepsilon j)|^2\,dx-\int_{\varepsilon i}^{\varepsilon(i+1)}\int_{\varepsilon j}^{\varepsilon(j+1)}|\partial_x\psi_\mu(x,y)|^2\,dydx\right|\\
	&\,=\left|\int_{\varepsilon i}^{\varepsilon(i+1)}\int_{\varepsilon j}^{\varepsilon(j+1)}\left(|\partial_x\psi_\mu(x,\varepsilon j)|^2-|\partial_x\psi_\mu(x,y)|^2\right)\,dxdy\right|\\
	&\,\lesssim \varepsilon\|\partial_{yx}^2\psi_\mu\|_{L^2(Q_{ij}^\varepsilon)}\|\partial_x\psi_\mu\|_{L^2(Q_{ij}^\varepsilon)}\,,
	\end{split}
	\]
	and analogously
	\[
	\left|\varepsilon\|u_\varepsilon'\|_{L^2(h_{ij}^\varepsilon)}^2-\|\partial_y\psi_\mu\|_{L^2(Q_{ij}^\varepsilon)}^2\right|\lesssim \varepsilon\|\partial_{xy}^2\psi_\mu\|_{L^2(Q_{ij}^\varepsilon)}\|\partial_y\psi_\mu\|_{L^2(Q_{ij}^\varepsilon)}
	\]
	with $h_{ij}^\varepsilon:=\left\{\varepsilon i\right\}\times[\varepsilon j,\varepsilon(j+1)]$. 
	
	Now, since $\psi_\mu\in H^2(H_\theta^+)$, if we denote by $I_\varepsilon\subset\Z^2$ the set of indices $(i,j)\in\Z^2$ for which $Q_{ij}^\varepsilon\cap s_\theta\neq\emptyset$ and we set $N_\varepsilon^+:=H_\theta^+\cap\left(\bigcup_{(i,j)\in I_\varepsilon}Q_{ij}^\varepsilon\right)$, then summing the previous estimates over all $Q_{ij}^\varepsilon$ fully contained in $H_\theta^+$ we obtain
	\[
	\left|\varepsilon\|u_\varepsilon'\|_{L^2\left(\left(\G_\varepsilon\cap H_\theta^+\right)\setminus N_\varepsilon^+\right)}^2-\|\nabla\psi_\mu\|_{L^2(H_\theta^+\setminus N_\varepsilon^+)}^2\right|\lesssim\varepsilon\|\psi_\mu\|_{H^2(H_\theta^+)}^2\,.
	\]
	As the same computations can be repeated replacing $H_\theta^+$ with $H_\theta^-$, we have
	\begin{equation}
	\label{eq:dereps}
	\left|\varepsilon\|u_\varepsilon'\|_{L^2(\G_\varepsilon\setminus N_\varepsilon)}^2-\|\nabla\psi_\mu\|_{L^2(\R^2\setminus N_\varepsilon)}^2\right|\lesssim\varepsilon\left(\|\psi_\mu\|_{H^2(H_\theta^+)}^2+\|\psi_\mu\|_{H^2(H_\theta^-)}^2\right)\,,
	\end{equation}
	where $N_\varepsilon:=\bigcup_{(i,j)\in I_\varepsilon}Q_{ij}^\varepsilon$. Clearly, an analogous argument also gives
	\[
	\left|\varepsilon\|u_\varepsilon'\|_{L^2(\G_\varepsilon\cap N_\varepsilon)}^2-\|\nabla\psi_\mu\|_{L^2(N_\varepsilon)}^2\right|\lesssim\varepsilon\left(\|\psi_\mu\|_{H^2(H_\theta^+\cap N_\varepsilon)}^2+\|\psi_\mu\|_{H^2(H_\theta^-\cap N_\varepsilon)}^2\right)\,,
	\]
	where the only possible subtlety arises when $s_\theta$ coincides with a vertical or a horizontal line, but can be nevertheless directly managed exploiting \eqref{eq-mixedsecond} (as before).
	As a consequence, the third line of \eqref{eq:GSthetaG} follows by \eqref{eq:dereps} and $\|\nabla\psi_\mu\|_{L^2(N_\varepsilon)}\to0$ as $\varepsilon\to0$ (this latter fact being ensured by $\psi_\mu\in H^1(\R^2)$ and the absolute continuity of the Lebesgue integral).
 \end{proof}
 
\begin{lem}
	\label{l.eps.bound-Z}
	For every $p\in(2,4)$, $q\in(2,3)$ and $\mu>0$ there exists $M>0$ such that
	\[
	\frac{1}{M}\le \varepsilon\|u^\prime_\varepsilon\|_{2,\varepsilon}^2\,,\,\varepsilon\|u_\varepsilon\|_{p,\varepsilon}^p\,,\,\varepsilon\sum_{i\in\Z}|u_\varepsilon(\varepsilon i\vec{v})|^q\le M
	\]
	for every $\varepsilon>0$ and every ground state $u_\varepsilon$ of $\E_{\G_\varepsilon}\big(\frac{2\mu}{\varepsilon}\big)$.
\end{lem}

\begin{proof}
	Since for this choice of $p,q$ and $\mu$ Theorem \ref{thm:GSZ} yields $E(u_\varepsilon, \G_\varepsilon)<0$, we have
	\begin{equation}
	\label{eq:negE2}
	 \frac{\varepsilon}{2}\|u_\varepsilon'\|_{2,\varepsilon}^2<\frac{\varepsilon}{2p}\|u_\varepsilon\|_{p,\varepsilon}^p+\frac{\varepsilon|\vec{v}|}{q}\sum_{i\in\Z}|u_\varepsilon(\varepsilon i\vec{v})|^q. 
	\end{equation}
	By \eqref{GN-eps.2}, there exists $C_p>0$ depending only on $p$ such that
	\begin{equation}
		\label{proof.eps-bound.p-Z}
		\frac{\varepsilon}{2p}\|u_\varepsilon\|_{p,\varepsilon}^p\le C_p\varepsilon^{\frac{p}{2}}\|u_\varepsilon\|_{2,\varepsilon}^2\|u_\varepsilon^\prime\|_{2,\varepsilon}^{p-2}=2C_p\mu\left(\varepsilon\|u_\varepsilon^\prime\|_{2,\varepsilon}^2\right)^{\frac{p}{2}-1}\,,
	\end{equation}
    whereas combining Lemma \ref{l.GN-grid.Zper.}, Lemma \ref{l.GN-grid.Zper.0} and \eqref{GN-eps.2} gives
	\begin{multline}
		\label{proof.eps-bound.q-Z}
		\varepsilon\sum_{i\in\Z}|u_\varepsilon(\varepsilon i\vec{v})|^q\lesssim\|u_\varepsilon\|_{L^q(\G_\varepsilon')}^q+\varepsilon\|u_\varepsilon\|_{L^{2(q-1)}(\G_\varepsilon')}^{q-1}\|u_\varepsilon'\|_{L^2(\G_\varepsilon')}\\
		\lesssim\varepsilon^{\frac q2}\|u_\varepsilon\|_{2,\varepsilon}\|u_\varepsilon'\|_{2,\varepsilon}^{q-1}\lesssim (\varepsilon\|u_\varepsilon^\prime\|_{2,\varepsilon}^2)^\frac{q-1}{2}\,,
	\end{multline}
	with $\G_\varepsilon'$ defined by \eqref{eq-Gprime}. Coupling \eqref{eq:negE2}, \eqref{proof.eps-bound.p-Z} and \eqref{proof.eps-bound.q-Z} with $p<4$ and $q<3$ yields 
	\[
	\varepsilon\|u^\prime\|_{2,\varepsilon}^2\,,\,\varepsilon\|u_\varepsilon\|_{p,\varepsilon}^p\,,\,\varepsilon\sum_{i\in\Z}|u_\varepsilon(\varepsilon i \vec{v})|^q\lesssim 1. 
	\]
	As for the lower bound, note that by Lemma \ref{l.en.upperbound-Z} there exists $K>0$ such that, as $\varepsilon\to0$, 
	\[ \frac{\varepsilon}{2}\|u_\varepsilon^\prime\|_{2,\varepsilon}^2-\frac{\varepsilon}{2p}\|u_\varepsilon\|_{p,\varepsilon}^p-\frac{\varepsilon|\vec{v}|}{q}\sum_{i\in \Z}|u_\varepsilon(\varepsilon i\vec{v})|^q<-K\,.
	\]
	Moving from this and arguing as in the proof of Lemma \ref{lem:aprZ2} one obtains
	\[
	\varepsilon\|u^\prime_\varepsilon\|_{2,\varepsilon}^2\,,\,\varepsilon\sum_{i\in\Z}|u_\varepsilon(\varepsilon i \vec{v})|^q\gtrsim1\,.
	\]
	To conclude, assume by contradiction that
	\[ 
	\liminf_{\varepsilon\to0^+}\varepsilon\|u_\varepsilon\|_{p,\varepsilon}^p=0\,, 
	\]
	so that, by Lemma \ref{l.en.upperbound-Z},
	\begin{equation}
	\label{eq:ass1}
	\E_{\R^2,\theta}(\mu)\geq\liminf_{\varepsilon\to0}\varepsilon\E_{\G_\varepsilon}\left(\frac{2\mu}\varepsilon\right)=\liminf_{\varepsilon\to0}\left(\frac\varepsilon2\|u_\varepsilon'\|_{2,\varepsilon}^2-\frac{\varepsilon|\vec{v}|}q\sum_{i\in\Z}|u_\varepsilon(\varepsilon i \vec{v})|^q\right)\,.
	\end{equation}
	Now, let $\mathcal{A}u_\varepsilon$ be the usual piecewise-affine extension to $\R^2$ of $u_\varepsilon$. Since $\mathcal{A}u_\varepsilon$ and $u_\varepsilon$ coincide on $\V_{\G_\varepsilon}$, adapting to $\mathcal{A}u_\varepsilon$ the computations performed with $\psi_\mu$ in the proof of Lemma \ref{l.en.upperbound-Z} to establish the second line of \eqref{eq:GSthetaG}, in view of Lemma \ref{l.eps.bound.grid.es.Zper.}, and the already proved upper bound on $\varepsilon\|u_\varepsilon'\|_{2,\varepsilon}^2$, one obtains
	\begin{equation}
	\label{eq:delteAueps}
	\|\mathcal{A}u_\varepsilon\|_{L^q(s_\theta)}^q=\varepsilon|\vec{v}|\sum_{i\in\Z}|u_\varepsilon(\varepsilon i\vec{v})|^q+o(1)\qquad\text{as}\quad\varepsilon\to0\,.
	\end{equation}
 	Since (arguing exactly as in \cite[Lemma 6.1]{D24}) the upper bound on $\varepsilon\|u_\varepsilon'\|_{2,\varepsilon}^2$ is enough to obtain also 
 	\begin{equation}
 		\label{eq:normAZ}
 		\left|\|\mathcal{A}u_\varepsilon\|_{L^r(\R^2)}^r-\frac\varepsilon2\|u_\varepsilon\|_{r,\varepsilon}^r\right|\lesssim\varepsilon\qquad\text{as}\quad\varepsilon\to0
 	\end{equation}
 	for every $r\geq2$, combining with \eqref{eq:ass1} and \eqref{eq:gradA} entails
 	\[
 	\E_{\R^2,\theta}(\mu)\geq\liminf_{\varepsilon\to0}\left(\frac12\|\nabla\mathcal{A}u_\varepsilon\|_2^2-\frac1q\|\mathcal{A}u_\varepsilon\|_{L^q(s_\theta)}^q\right)\geq\liminf_{\varepsilon\to0}\inf_{v\in H_\mu^1(\R^2)}\left(\frac12\|\nabla v\|_2^2-\frac1q\|v\|_{L^q(s_\theta)}^q\right)\,.
 	\]
	This provides the contradiction we seek, since the proof of Theorem \ref{thm:exGS_theta} above extends verbatim to show that 
	\[
	\inf_{v\in H_\mu^1(\R^2)}\left(\frac12\|\nabla v\|_2^2-\frac1q\|v\|_{L^q(s_\theta)}^q\right)
	\]
	is attained for every $q\in(2,3)$ and $\mu>0$, and, letting $w\in H_\mu^1(\R^2)$ be an associated minimizer, one obtains
	\[
	\E_{\R^2,\theta}(\mu)\leq E_\theta(w,\R^2)<\frac12\|\nabla w\|_2^2-\frac1q\|w\|_{L^q(s_\theta)}^q=\inf_{v\in H_\mu^1(\R^2)}\left(\frac12\|\nabla v\|_2^2-\frac1q\|v\|_{L^q(s_\theta)}^q\right)\,.
	\]
\end{proof}

We can now conclude the proof of Theorem \ref{thm:limZ_theta}.

\begin{proof}[Proof of Theorem \ref{thm:limZ_theta}]
	First, let $V$ be as in \eqref{eq:Vsimpl}. In this case, the claim follows arguing exactly as in the proof of Theorem \ref{thm:limZ2}. Indeed, by Lemma \ref{l.eps.bound.grid.es.Zper.}, \eqref{eq:gradA}, \eqref{eq:delteAueps} and \eqref{eq:normAZ}, if $u_\varepsilon\in H_{\frac{2\mu}{\varepsilon}}^1(\G_\varepsilon)$ is a ground state of $\E_{\G_\varepsilon}\big(\frac{2\mu}{\varepsilon}\big)$, then it follows that 
	\[
	w_\varepsilon=\sqrt{\frac\mu{\|\mathcal{A}u_\varepsilon\|_2^2}}\,\mathcal{A}u_\varepsilon
	\] 
	is a minimizing sequence in $H_\mu^1(\R^2)$ for $E_{\theta}(\cdot,\R^2)$, and to conclude it is enough to show that, possibly after a proper translation, it converges strongly to some $\psi_\mu\in H_\mu^1(\R^2)$ such that $E_\theta(\psi_\mu,\R^2)=\E_{\R^2,\theta}(\mu)$, which can be done repeating verbatim the argument in the proof of Theorem \ref{thm:exGS_theta} above.
	
	Consider now a general $\Z$-periodic set $V\subset\V_{\G_1}$, which we can assume to contain the origin of $\R^2$ with no loss of generality, and let $V_\varepsilon=\varepsilon V$. Note that to prove Theorem \ref{thm:limZ_theta} with this choice of $V_\varepsilon$ it is enough to show that
	\begin{equation}
	\label{eq:VtoVsimp}
	\varepsilon\left|\sum_{\vv\in  V_\varepsilon}|u_\varepsilon(\vv)|^q-\#V_0\sum_{i\in\Z}|u_\varepsilon(\varepsilon i\vec{v})|^q\right|=o(1)\qquad\text{as}\quad\varepsilon\to0
	\end{equation}
	for every $u_\varepsilon\in H_{\frac{2\mu}{\varepsilon}}^1(\G_\varepsilon)$ with $\varepsilon\|u_\varepsilon'\|_{2,\varepsilon}^2$ bounded from above uniformly on $\varepsilon$, where $\vec{v}$ is the vector associated to $V$ as in Definition \ref{def:Z-per} and $V_0$ is the set associated to $V$ as in Remark \ref{rem:Per_dec}. Indeed, if \eqref{eq:VtoVsimp} holds true, applying it to the ground states of $E(\cdot,\G_\varepsilon)$ in $H_{\frac{2\mu}{\varepsilon}}^1(\G_\varepsilon)$ with $V$ as in \eqref{eq:Vsimpl} allows to recover the upper bound in Lemma \ref{l.en.upperbound-Z} for $\E_{\G_\varepsilon}\big(\frac{2\mu}{\varepsilon}\big)$ with a general $V$, which can then be used to recover Lemma \ref{l.eps.bound-Z} and the rest of the proof of Theorem \ref{thm:limZ_theta} too.
	
	To prove \eqref{eq:VtoVsimp}, we recall Remark \ref{rem:Per_dec} and \eqref{eq-Gprime} to compute
	\[
	\begin{split}
		\varepsilon\left|\sum_{\vv\in V}|u_\varepsilon(\varepsilon \vv)|^q-\#V_0\sum_{i\in\Z}|u_\varepsilon(\varepsilon i\vec{v})|^q\right|&\,= \varepsilon\left|\sum_{i\in\Z}\sum_{\vv\in V_0+i\vec{v}}|u_\varepsilon(\varepsilon \vv)|^q-\#V_0\sum_{i\in\Z}|u_\varepsilon(\varepsilon i\vec{v})|^q\right|\\
		&\,=\varepsilon\left|\sum_{i\in\Z}\sum_{\vv\in V_0+i\vec{v}}\left(|u_\varepsilon(\varepsilon \vv)|^q-|u_\varepsilon(\varepsilon i\vec{v})|^q\right)\right|\\
		&\,\lesssim\varepsilon\sum_{i\in\Z}\sum_{\vv\in V_0+i\vec{v}}\int_{\varepsilon(Q_0+i\vec{v})}|u_\varepsilon|^{q-1}|u_\varepsilon'|\,dx\\
		&\,=\varepsilon\#V_0\int_{\G_\varepsilon'}|u_\varepsilon|^{q-1}|u_\varepsilon'|\,dx\lesssim\varepsilon\|u_\varepsilon\|_{L^{2(q-1)}(\G_\varepsilon')}^{q-1}\|u_\varepsilon'\|_{L^2(\G_\varepsilon')}\,.\end{split}
	\]
	Since by Lemma \ref{l.GN-grid.Zper.0} we obtain
	\[
	\|u_\varepsilon\|_{L^{(q-1)}(\G_\varepsilon')}^{2(q-1)}\lesssim\varepsilon^{q-1}\|u_\varepsilon\|_{2,\varepsilon}\|u_\varepsilon'\|_{2,\varepsilon}^{2q-3}=\sqrt{2\mu}(\varepsilon\|u_\varepsilon'\|_{2,\varepsilon}^2)^{q-\frac32}\lesssim1\,,
	\]
	plugging into the estimate above leads to
	\[
	\varepsilon\left|\sum_{\vv\in V}|u_\varepsilon(\varepsilon \vv)|^q-\#V_0\sum_{i\in\Z}|u_\varepsilon(\varepsilon i\vec{v})|^q\right|\lesssim\varepsilon\|u_\varepsilon'\|_{2,\varepsilon}\lesssim\sqrt{\varepsilon}=o(1)
	\] 
	as $\varepsilon\to0$.
\end{proof}

\subsection{The limit problem $\E_{\R^2,\theta,R}$: proof of Theorem \ref{thm:limZ_S}}
The argument is completely analogous to those already developed for Theorems \ref{thm:limZ2}--\ref{thm:limZ_theta}. The line of the proof can then be repeated with no significant changes, just keeping in mind that, to compare the $L^q$ norm on the strip $S_{\theta,R}$ of a function in $\R^2$ with the corresponding sum of $q$-powers of the values of the function at the vertices in $V_\varepsilon$ as in the hypotheses of Theorem \ref{thm:limZ_S}, it is enough to argue exactly as in the proof of Lemma \ref{l.GN-grid-i}.

\section*{Statements and Declarations}

\noindent\textbf{Conflict of interest}  The authors declare that they have  no conflict of interest.

\bigskip

\noindent\textbf{Acknowledgements.} D.B, S.D. and L.T. acknowledge that this study was carried out within the project E53D23005450006 ``Nonlinear dispersive equations in presence of singularities'' -- funded by European Union -- Next Generation EU within the PRIN 2022 program (D.D. 104 - 02/02/2022 Ministero dell'Universit\`a e della Ricerca). This manuscript reflects only the author's views and opinions and the Ministry cannot be considered responsible for them.

\end{document}